\documentclass[12pt]{amsart}

\usepackage{graphicx}
\usepackage{eucal}
\usepackage{xcolor}
\usepackage{amsfonts,amssymb,amscd}

\newtheorem{theorem}{Theorem}[section]

\theoremstyle{definition}
\newtheorem{definition}[theorem]{Definition}
\newtheorem{example}[theorem]{Example}
\newtheorem{proposition}[theorem]{Proposition}
\newtheorem{corollary}[theorem]{Corollary}
\newtheorem{observation}[theorem]{Observation}

\theoremstyle{remark}

\numberwithin{equation}{section}
\usepackage{graphicx}
\usepackage{hyperref}
\usepackage{geometry}
\usepackage[all]{xy}

\begin{document}
\title{Immersed flat ribbon knots}
\date{\today}
\author{Jos\'{e} Ayala}

\address{Facultad de Ciencias, Universidad Arturo Prat, Iquique, Chile} 
\address{School of Mathematics and Statistics, University of Melbourne
              Parkville, VIC 3010 Australia}
              \email{jayalhoff@gmail.com}
\author{David Kirszenblat}
\address{School of Mathematics and Statistics, University of Melbourne
              Parkville, VIC 3010 Australia}
\email{d.kirszenblat@student.unimelb.edu.au}
\author{Joachim Hyam Rubinstein}
\address{School of Mathematics and Statistics, University of Melbourne
              Parkville, VIC 3010 Australia}
\email{joachim@unimelb.edu.au}
\maketitle
\baselineskip=20 true pt
\maketitle \baselineskip=1.1\normalbaselineskip


\begin{abstract} 

We study the minimum ribbonlength for immersed planar ribbon knots and links. Our approach is to embed the space of such knots and links into a larger more tractable space of disk diagrams. When length minimisers in disk diagram space are ribbon, then these solve the ribbonlength problem. We also provide examples when minimisers in the space of disk diagrams are not ribbon and state some conjectures. We compute the minimal ribbonlength of some small knot and link diagrams and certain infinite families of link diagrams. Finally we present a bound for the number of crossings for a diagram yielding the minimum ribbonlength of a knot or link amongst all diagrams. 
\end{abstract}

\section{Introduction}

A flat ribbon knot is an immersion of an annulus into the Euclidean plane such that the core of the annulus corresponds to an immersed knot diagram. Flat ribbon links then are a finite union of flat ribbon knots. We require that the combinatorics of the self-intersections of the knot or link match that of the associated immersed annulus or annuli. The ribbonlength problem aims to find the minimal ratio between the length of the core to the ribbon width over all the planar realisations in a knot or link type. We normalise our ribbons so that the ribbon width is $2$ as can be done by a suitable homothety of the plane. 

Kauffman \cite{kauffman} proposed a model for folded flat ribbon knots whose core is piecewise linear. He gave conjectures about the ribbonlength of the trefoil and the figure-eight knot. Currently, results in the literature provide bounds on the ribbonlength of flat folded ribbons knots for certain families of torus knots. The ribbonlength problem remains open for general immersed and folded ribbon knots. A good survey on flat folded ribbon knots can be found in \cite{denne1}. Denne, Sullivan and Wrinkle have reported to be working on the same problem considered here \cite{denne3}. 
 
 A similar scenario occurs for the ropelength problem, the 3-dimensional analog of the ribbonlength problem \cite{sullivan1, diao, gonzalez, litherland}. Minimum ropelength is the smallest ratio between the thickness of a tubular neighbourhood of a knot or link and its length. In this case, the unknot is minimised by a round circle and certain chain links are minimised by linked piecewise $C^2$ loops \cite{sullivan1}.
 

\begin{figure} [h!]
\centering
\includegraphics[width=1\textwidth]{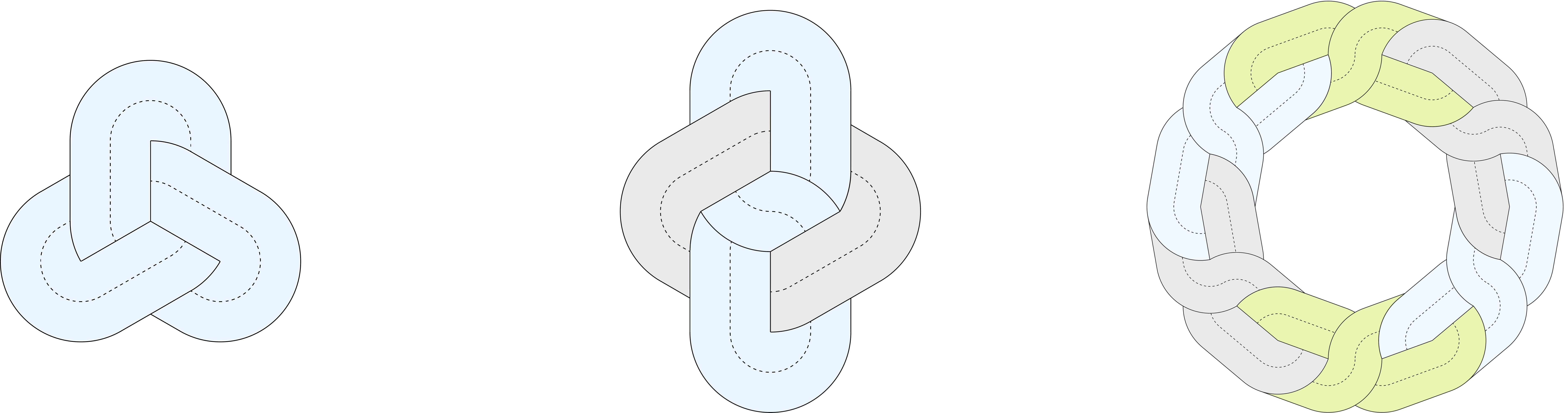}
\caption{Examples of minimal ribbonlength knot and link diagrams. From left to right: the trefoil, Hopf link and an element of an infinite family of link diagrams whose ribbonlength can be computed.}
\label{fig:examples}
\end{figure}
 
In this note, we describe a method that in some cases leads to the solution of the ribbonlength problem for flat immersed ribbon knot and link diagrams, see Fig. \ref{fig:examples}. Our method finds the tightest immersed flat ribbon knot or link diagram satisfying some simple geometric constraints. These constraints allow us to embed the space of immersed flat ribbon knots and link diagrams into a larger space of disk diagrams. The latter space has better properties and in particular, passing to a minimum length is straightforward. Moreover, we can characterise critical points of length in this larger space as piecewise $C^2$ immersions, formed by concatenations of arcs of (unit radius) circles and straight line segments.  

When a minimum length disk diagram satisfies the conditions for being a ribbon knot or link, this immediately gives us the solution for the ribbonlength problem for such diagrams. We give examples where this is the case and also where the minimum length disk diagrams are not ribbon. We compute the minimal ribbonlength of some small knot and link diagrams like the unknot, trefoil, Hopf link and an infinite family of link diagrams - see Figure \ref{fig:examples}.

Finally using disk diagrams, we are able to give a bound on the number of crossings in a diagram of a given knot or link which minimises ribbonlength amongst all diagrams.

\section{Immersed ribbon loops}

An immersion is a $C^1$ map between manifolds whose Jacobian is injective everywhere. We consider $C^1$ immersed loops $\gamma: S^1\to\mathbb R^2$, where the length of the derivative $||\gamma^\prime||=1/\ell$ and $S^1=[0,1]/\{0\sim 1\}$, denoting the length of the loop by $\ell$. We endow the space of immersed loops with the $C^1$ metric. A map $\gamma: X\to Y$ is considered either as a function or its image $\gamma (X) \subset Y$ when no confusion occurs. 

The medial axis of an embedded planar loop is the locus of the centres of the circles of maximal radii which are tangent to the loop and are inscribed in the compact region bounded by the loop \cite{younes}. Note that the medial axis is contained in the focal cut locus, i.e the points in the region bounded by the loop where either the distance function to the loop has zero Jacobian or is not smooth. This locus consists of either focal conjugate points for the normal map to the loop (where the normal map has a singular Jacobian) or focal cut points where there are several equal length minimising normals to the loop. A simple example is an ellipse where the medial axis or focal cut locus is a closed interval. The end points of this interval are focal conjugate points and the interior points are focal cut points. 

In this paper, we will use an analogous concept of medial axis for an embedded planar open or closed arc. Note that the key difference to the case of an embedded loop is that such an arc does not separate the plane.  So we do not restrict attention to one side of the arc. In particular, the medial axis of an arc is not necessarily connected and there can be components associated to normals on either side of the arc. So the medial axis for an arc is the locus of centres of maximal circles touching the arc with interiors disjoint from the arc and is contained in the focal cut locus of the arc. The medial axis of a closed arc is necessarily closed. Note we will combine the medial axes of a finite closed arc cover of an immersed loop, obtaining a closed set as a result. 

\begin{definition} Given a $C^1$ immersion $\gamma:S^1 \to \mathbb R^2$, choose $I\subset S^1$ a sufficiently small open or closed interval such that $\gamma(I)$ is embedded. The {\bf separation bound} for $\gamma$ on $I$ is the condition that for an arbitrary point $c$ on the medial axis of $\gamma$ restricted to $I$, $||\gamma(s)-c||\geq 1$, for all $s\in I$. 
\end{definition}

\begin{definition} Given a $C^1$ immersion $\gamma:S^1 \to \mathbb R^2$,  {\bf the weak separation bound} is the condition that for each point $p = \gamma(s)$, there are two circles of radius $1$ which touch $\gamma$ at $p$ one on each side of $\gamma$. In other words, the circles pass through $p$, have a common tangent to $\gamma$ at $p$ and locally lie on one side and the other side of $\gamma$ near $p$.

\end{definition}

Let $T(s)$ and $N(s)$, for $s\in S^1$, be a continuous choice of unit length tangent and normal vectors to $\gamma$ respectively at $\gamma(s)$.  Let $L_T(s)=\langle \gamma(s)+tT(s) \rangle$ and $L_N(s)=\langle \gamma(s)+tN(s) \rangle$ be the tangent and normal lines passing through $\gamma(s)$. We define a continuous map $\Gamma :S^1 \times (-1,1) \to \mathbb R^2$, which is an extension of $\gamma$, by $\Gamma(s,t)=\gamma(s)+tN(s)$. So $\Gamma(s,0)=\gamma(s)$, for $s\in S^1$. The notation $\Gamma$ will be considered either an immersion of an open annulus or its image when no confusion occurs, assuming that $\gamma$ satisfies the separation bound for a collection of small intervals covering $\gamma$.

\begin{proposition}\label{localribbon}  Let $\gamma: S^1 \to\mathbb R^2$ be an immersed loop and $I \subset S^1$ be a closed interval so that $\gamma(I)$ is embedded and $\gamma$ satisfies the separation bound for $\gamma$ on $I$. If $I$ is sufficiently small then $\Gamma(I\times(-1,1))$ is embedded.
\end{proposition}
\begin{proof} Suppose by way of contradiction that $\Gamma(s,t)=\Gamma(s',t')$, where $(s,t) \ne (s',t')$, $s,s' \in I$ and $0 \le t,t'<1$. Firstly, we cannot have $t=t'=0$ since $\gamma(I)$ is embedded. Next we cannot have $t=t' \ne 0$ since this would contradict the separation bound. In fact, since $I$ is closed, in this case we can find $t=t' \ne 0$ minimal so that $\Gamma(s,t)=\Gamma(s',t)$ is a point on the medial axis at a distance at most $t<1$ from $\gamma$. 

So it remains to deal with the case that $t \ne t'$. Consider the smallest value of min$\{t,t'\}$ for which  $\Gamma(s,t)=\Gamma(s',t')$ for $s,s' \in I$. (We need $I$ closed to find such a smallest value). Consider the angle between the normal lines $L_N(s), L_N(s')$ at $X=\Gamma(s,t)=\Gamma(s',t')$.  If the angle is zero then $s=s'$ and this is impossible as the normal line is embedded. Since  $I$ is chosen as a small closed interval, the angle between the normal lines will also be small. Now we can perturb $X$ to $\hat X$ along the bisector of this angle. By continuity, we claim that we can find $(\hat s, \hat t)$ near $(s,t)$ and $(\hat s', \hat t')$ near $(s',t')$ so that $\hat X=\Gamma(\hat s, \hat t)=\Gamma(\hat s',\hat t')$ and $\hat t <t, \hat t' <t'$.
 
 The claim follows by analysing the distance to the curve $\gamma$. Note that as $I$ is small, $\hat s, \hat s'$ will be between $s,s'$ in $I$. The reason is by the weak separation bound, there are two circles of radius one which touch $\gamma$ at either of $\gamma(s), \gamma(t)$ (see Figure ?). In particular, $\gamma$ is locally between these circles near $\gamma(s), \gamma(t)$ respectively. Hence if we choose an initial point $X$ lying on the line segment $L$ joining these two circles and perturb it in the direction of a small angle to this line, clearly the closest point on $\gamma$ to the perturbed point will move in the same direction away from $L$ as the perturbation of $X$. Hence since we have chosen the perturbation as the bisector of the angle between the normal lines to $\gamma(s), \gamma(t)$, we see that $\hat s, \hat s'$ will be between $s,s'$ in $I$.

But this contradicts our choice since min$\{\hat t,\hat t'\}$ $<$ min$\{t,t'\}$. So this completes the proof. 
\end{proof}

\begin{definition}\label{crossingcond} Assume that $\gamma$ is an immersed loop satisfying the separation bound for each of a collection of closed intervals covering $S^1$. The  {\bf crossing condition} is the following requirements:
\begin{itemize}
\item $\gamma$ and $\Gamma$ have only double points of self-intersection. In particular, $\gamma$ self-intersects in isolated double points or double intervals (for the latter, the self-intersections are not transverse). Also, for any isolated double point $p$ (or double interval $J$) of $\gamma$, there is a neighbourhood $U_p$ (or $U_J$) so that $\Gamma^{-1}(U_p)=D_1\cup D_2$ (or $\Gamma^{-1}(U_J)=D_1\cup D_2$), where $D_1, D_2$ are disjoint open disks,  see Fig. \ref{fig:ribcross} and Fig. \ref{fig:annulus}.
\item The two branches of $\gamma$ cross at $p$ or $J$, see Fig. \ref{fig:ribcross}.
\item $S^1 \times (-1,1) \setminus D_i$ is contractible, for $i=1,2$. Moreover $S^1 \times (-1,1) \setminus {\mathcal D}$ has $n$ contractible components, where $\mathcal D$ consists of $n$ pre-images of neighbourhoods of all isolated double points and double intervals for $\gamma$. 
\item $D_i \cap (S^1 \times \{0\})$ is connected for $i=1,2$. So $U_p \cap \gamma$ (or $U_J \cap \gamma$) consists of two intersecting open intervals.
\item The pre-images of all double points of $\Gamma$ come from such pairs of disks $D_1, D_2$.
\end{itemize} 
\end{definition}


\begin{figure} [h!]
\centering
\includegraphics[width=1\textwidth]{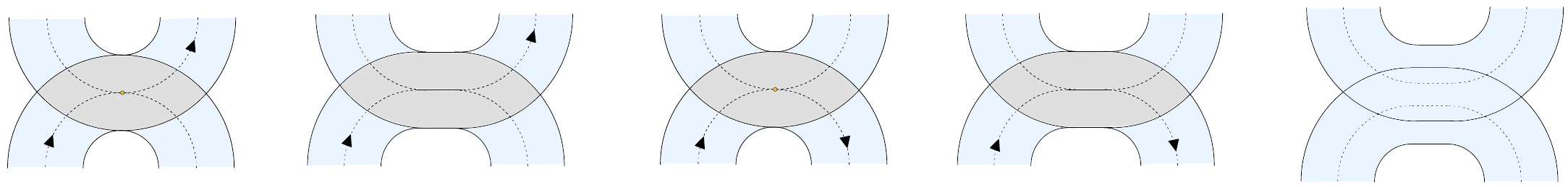}
\caption{The broken lines are the ribbon loop $\gamma$ (core) associated with the open ribbon $\Gamma$, while the arrows show a branch of $\gamma$. The grey areas are neighbourhoods around self-intersections. From left to right: a ribbon self-intersecting at an isolated double point. A ribbon self-intersecting in a double interval. A non-crossing self-intersection at a double point. A non-crossing self-intersection in a double interval. Finally, a self-intersection not satisfying the crossing condition as the ribbon self-intersects but there are no associated double points or intervals of the core.}
\label{fig:ribcross}
\end{figure}

\begin{definition} A $C^1$ immersed loop $\gamma: S^1\to\mathbb R^2$ satisfying the crossing condition and the separation bound for a collection of closed intervals covering $S^1$ is called a  {\bf ribbon loop}. The immersed open annular neighbourhood $\Gamma(S^1 \times (-1,1))$ of a ribbon loop $\gamma(S^1)$ is called the ribbon associated to the loop. A finite collection of $C^1$ loops which together satisfy these two conditions is called a ribbon link. The space of ribbon knots and links has the standard $C^1$ topology. 
\end{definition}


\begin{observation}\label{obs} \hfill 
\begin{enumerate}
\item \label{dcsd} {\bf Critical self-distance}. Suppose that the normal lines $L_N(s), L_N(t)$ coincide for $s \neq t$. Since $\gamma'(s)$ and $\gamma'(t)$ are parallel, we have that $\gamma'(s)\cdot(\gamma(s)-\gamma(t))= \gamma'(t)\cdot(\gamma(s)-\gamma(t))=0$. This gives a double critical self-distance $\|\gamma(s)-\gamma(t)\|$  between $\gamma(s), \gamma(t)$, \cite{litherland}. We observe that a bound on medial axis distance implies a bound on double critical self-distance  between such points $\gamma(s), \gamma(t)$. 

\item \label{curvature}  {\bf Bound on curvature}. If there is a maximal radius circle with centre $c$ which is tangent to $\gamma$ at $\gamma(s)$ satisfying $||\gamma(s)-c||<1$, then this implies that $\gamma$ does not satisfy the separation condition. This implies that the curvature $\kappa$ of $\gamma$ is at most $1$ at all points where a ribbon loop or link $\gamma$ is $C^2$. 

Moreover, ribbon loops or links are $C^{1,1}$ curves. To see this, note for $s$ fixed, if there are sufficiently close points $t \in S^1$ to $s$ satisfying $||\gamma^\prime(s)-\gamma^\prime(t)|| >k|s-t|/\ell $ with $k>1$, then the maximal radius circle touching the loop at $\gamma(s)$ has radius smaller than $1$, which again contradicts the separation condition. So this establishes that ribbon loops are $C^{1,1}$ curves. 

Finally the same argument shows that $C^1$ immersed loops or links satisfying the weak separation bound are also in the class $C^{1,1}$. This will be important when we discuss loops in disk space. 

\item {\bf Non-overlapping condition}. Two branches of the open annular neighbourhood $\Gamma$ of a ribbon loop or link $\gamma$ intersect only if their core components of $\gamma$ intersect, see Fig. \ref{fig:ribcross}.

\item {\bf Complementary regions}. A ribbon loop or link separates the plane into open complementary connected domains each of whose boundary is a piecewise $C^1$ polygon. 
\end{enumerate}
\end{observation}


 \begin{figure} [h!]
\centering
\includegraphics[width=1\textwidth]{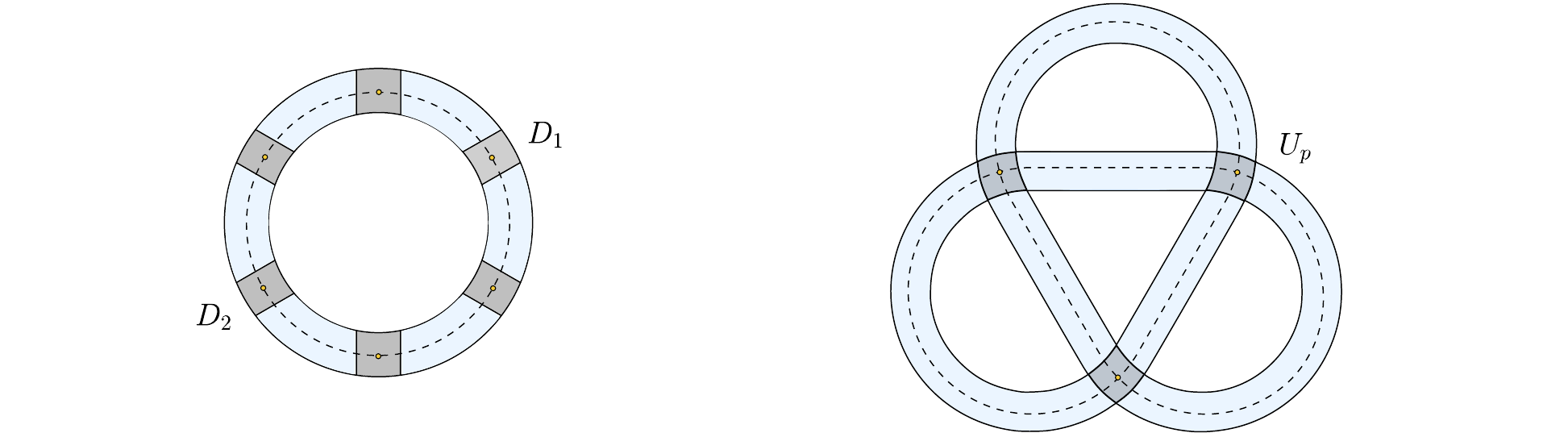}
\caption{For a double point $p$ of $\gamma$ there is a neighbourhood $U_p$ so that $\Gamma^{-1}(U_p)=D_1\cup D_2$, where $D_1, D_2$ are disjoint open disks which separate the annulus into contractible components. Note that $\Gamma(s,0)=\gamma(s)$, for $s\in S^1$.}
\label{fig:annulus}
\end{figure}
 
\begin{proposition} \label{disk} Each complementary region of a ribbon loop or link contains an open disk with radius 1.
\end{proposition}

\begin{proof} Each complementary region $P$ of a ribbon loop or link $\gamma$ is a union of open domains each of whose boundary is a piecewise $C^1$ polygon.  Consider the ribbon $\Gamma$ associated with $\gamma$. (Recall we abuse notation by denoting mappings and images by the same symbol, when this does not create confusion). It is clearly sufficient to show that $P \setminus \Gamma$ is non empty. For if we can find a point $c \in P \setminus \Gamma$ then the open disk of radius $1$ centred at $c$ is clearly in $P$, since $\Gamma$ contains all points distance $<1$ from $\gamma$.  

Suppose by way of contradiction that $P \subset \Gamma$. We claim this contradicts the crossing condition satisfied by $\gamma$. Note that the crossing condition implies that the image of the ribbon $\Gamma$ deformation retracts onto the ribbon loop or link $\gamma$, see Fig. \ref{fig:ribcross} and Fig. \ref{fig:annulus}. In fact the crossing condition shows that the image ribbon can be viewed as an open regular neighbourhood of the image loop or link. Hence there is an isomorphism between first homology groups $H_1(\gamma, \mathbb Z)\to H_1(\Gamma, \mathbb Z)$ induced by inclusion, where we are identifying maps with their images in the plane. But $P$ fills in a non-zero homologous cycle of $\gamma$, namely $\partial P$. So if $P \subset \Gamma$ then the first homology $H_1(\Gamma)$ would not be isomorphic to $H_1(\gamma)$. So this contradiction establishes the result. 

\end{proof}

\begin{observation} \label{link} Proposition \ref{disk} does not extend to a finite collection of $C^1$ ribbon loops unless they collectively satisfy the ribbon conditions. 
\end{observation}

\begin{definition} A $C^1$ immersed planar loop or link is said to satisfy the {\bf double point condition} if it has finitely many isolated double points and double intervals.
\end{definition}

\begin{theorem}\label{doublepoint} Ribbon loops and links satisfy the double point condition. 
\end{theorem}

\begin{proof} Clearly the double points (or intervals) of a $C^0$ loop form a closed set in $S^1$. If there are infinitely many isolated double points or double intervals, then these accumulate at a double point or double interval. But this contradicts the crossing condition satisfied by ribbon loops. The explanation is the existence of disks $D_1,D_2$ associated with double points or double intervals (see Fig \ref{fig:annulus}) implies that the connected components of the set of double points in $S^1$ are the individual double points and double intervals, since pairs of such points and intervals are separated by open sets. (The closed subsets of the plane have the usual Hausdorff topology). Hence, there cannot be an accumulation point in this set. The same argument clearly applies to a ribbon link. 
\end{proof}

We immediately have the following.

\begin{corollary} \label{finitepartition}
A ribbon loop or link partitions the plane into a finite number of open disk regions with exactly one being unbounded. 
\end{corollary}

\begin{definition} \label{diskspace} Disk space $\mathcal D$ is the space of finite collections of $C^1$ immersed planar loops satisfying the following;

\begin{itemize}
\item the double point condition 
\item the weak separation condition
\item there is an open disk of radius 1 contained in each complementary region. Here complementary regions are obtained by taking the unions of complementary connected open domains which are adjacent at non-crossing isolated double points and intervals. (See Fig. \ref{fig:ribcross} \  third and fourth diagrams).

\end{itemize}
Also $\mathcal D$ has the standard $C^1$ topology. 
 \end{definition}
 
 The elements in $\mathcal D$ are called disk diagrams. In particular, a ribbon loop or link is an element of $\mathcal D$. 

\begin{observation}
A connected component of disk space $\mathcal D$ consists of all representations of a finite collection of $C^1$ immersed planar loops, satisfying the required properties, with the same combinatorial type of complementary regions.  Here the combinatorial type of a complementary region is defined by splitting the region open along all non-crossing double points and intervals using the resulting polygonal type. Equivalently, the loop or link can be perturbed to remove all non-crossing double points and intervals and the resulting complementary domains have the required combinatorial types.
\end{observation}

\begin{observation}
Note that as complementary regions contain open disks of radius 1, they cannot disappear under limits of sequences of loops. New non-crossing double points can appear or disappear under limits, but by definition, these do not change the combinatorial type of complementary regions.
\end{observation}

\section{Ribbonlength minimisers}

For a ribbon loop $\gamma: S^1\to \mathbb R^2$, we have defined the ribbon map $\Gamma:S^1\times(-1,1)\to \mathbb R^2$ whose image is also immersed, and abused notation by identifying a map with its image. We use similar notation for ribbon links. 

\begin{definition}

The combinatorial type of a ribbon loop or link is defined by the collection of complementary connected domains.  Two sets of domains are equivalent if there is a bijection between them preserving the boundary structure along the loop or link, where the corresponding polygonal types of the open complementary connected domains are isomorphic. By boundary structure, we mean divide the boundary into arcs bounded by vertices which are either isolated double points or ends of double intervals. 

\end{definition}

By convention, we use representative of a given ribbon diagram to mean any ribbon diagram with the same combinatorial type. 


\begin{definition} A ribbon diagram is $cs$ if its core is a finite $C^1$ concatenation of arcs of circles of radius 1 and line segments. The complexity of a $cs$ ribbon diagram is the number of arcs of circles of radius 1 plus the number of line segments. 
\end{definition}

\begin{theorem}\label{minimalribbon}
Suppose that $\gamma$ is a ribbon diagram. Then minimal length ribbon representatives of $\gamma$ exist.
\end{theorem}

\begin{proof} By Observation \ref{obs}, we know that the core of a ribbon knot or link diagram is in the class $C^{1,1}$. Therefore, we can apply the Arzela-Ascoli theorem and conclude that given a sequence $\gamma_i$ of ribbon representatives of $\gamma$, so that the lengths $||\gamma_i||$ converge to the infimum of lengths in the class of $\gamma$, there are convergent subsequences of the $\gamma_i$. The limit of such a subsequence is a loop or link $\hat \gamma$ which is $C^1$ and realises the least length in the class of $\gamma$.  

We will show that either $\hat \gamma$ satisfies the ribbon conditions and $\hat \gamma$ is a representative of $\gamma$ or this is true after a modification of $\hat \gamma$ which does not change its length. We will denote the subsequence converging to $\hat \gamma$ by $\gamma_i$ again. For the separation bound, note that none of the loops or links $\gamma_i$ have points in their medial axes closer than distance 1 to the loops or links. But then it is easy to deduce the same is true for $\hat \gamma$. To verify this, choose a sequence of maximal circles touching each $\gamma_i$ at a point $x_i$, so that the sequence $x_i$ converges to $x$ on $\hat \gamma$ and the circles are all on sides of $\gamma_i$ corresponding to a given side of $\hat \gamma$ at $x$. Then there is a subsequence of these circles which converges to a circle touching $\hat \gamma$ at $x$. So the radius of a maximal circle touching $\hat \gamma$ at $x$ is at least $1$ as claimed. This implies that $\hat \gamma$ satisfies the separation bound. 

For the crossing condition, we claim that double points of $\hat \Gamma$ are limits of double points for $\Gamma_i$. Notice that the domain and image of $\Gamma_i, \hat \Gamma$ are open sets. If there are points $(s,t) \ne (s',t')$ with $\hat \Gamma(s,t)=\hat \Gamma(s',t')$ then $\hat \Gamma(s,t) = \hat \Gamma(s',t') =lim \Gamma_i(s,t)=lim\Gamma_i(s',t')$. The image of $\Gamma_i$ contains a disk of some fixed radius about  $\Gamma_i(s,t)$ and similarly for $\Gamma_i(s',t')$, by convergence of $\Gamma_i$ to $\hat \Gamma$. Hence for $i$ large enough, $\Gamma_i(s,t), \Gamma_i(s',t')$ are sufficiently close that these disks overlap and we have double points of $\Gamma_i$ which converge to $(s,t),(s',t')$ as claimed. 

Next, we can take limits of subsequences of pairs of disks $D^i_1,D^i_2$  in the pre-image of $\Gamma_i$ corresponding to an isolated double point or double interval of $\gamma_i$ to get pairs of disks $\hat D_1, \hat D_2$ corresponding to double points or double intervals for $\hat \gamma$ in the pre-image of $\hat \Gamma$, establishing that $\hat \Gamma$ also satisfies one of the requirements for the crossing condition. This follows using Hausdorff convergence of subsets of the plane, see \cite{munkres}. However although the complements of the disks $\hat D_i$ are contractible, this does not follow necessarily for their limits $\hat D_1, \hat D_2$ - see Fig. \ref{fig:0123}b. To fix this issue, we will need to modify $\hat \gamma$ and this will be done later in the proof. 


\begin{figure} [h!]
\centering
\includegraphics[width=1\textwidth]{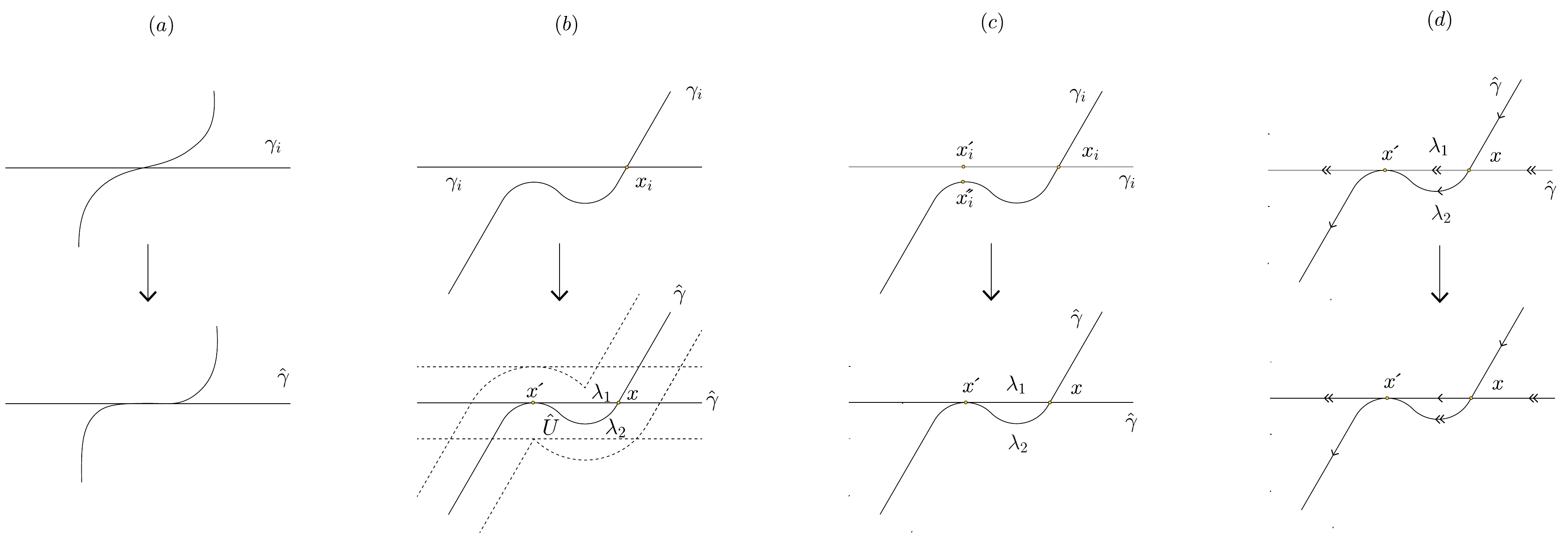}
\caption{(a) the limit of a sequence of double points for the $\gamma_i$ can lie in a double interval of $\hat \gamma$. (b) two ``close'' double points $x,x´$ for $\hat \gamma$. (c) the tangents to $\lambda_1, \lambda_2$ may agree at $x'$. (d) a switch between the arcs $\lambda_1$ and $\lambda_2$.}
\label{fig:0123}
\end{figure}

To continue the proof of the crossing condition, we need to show the double point set $\mathbb S$ of $\hat \gamma$ inside $\hat U=\hat \Gamma(\hat D_1)=\hat \Gamma(\hat D_2)$ is a single point or interval and also verify that $\hat D_i \cap (S^1 \times \{0\})$ is connected for $i=1,2$.  For the former, we claim that the double points of $\hat \gamma$ inside $\hat U$ are either isolated double points or double intervals. We can suppose by passing to a subsequence that the corresponding double points for $\gamma_i$ in $U_i=\Gamma_i(D^i_1)=\Gamma_i(D^i_2)$ are a convergent set of isolated double points or double intervals. We claim that $\mathbb S$ is either a double point or double interval which contains the limit of the sets of double points for the $\gamma_i$ in $U_i$. (Note that the limit of a sequence of double points for the $\gamma_i$ can lie in a double interval of $\hat \gamma$ - see Fig. \ref{fig:0123}a). So the problem is to show that $\mathbb S \cap \hat U$ is connected - it is clearly closed and contains the limit of the sets of double points for the $\gamma_i$ in $\hat U$. 
 
Suppose that $\mathbb S$ is not connected. Since $\mathbb S$ is closed, its complement in $\hat \gamma \cap \hat U$ is open. Then we can choose two arcs $\lambda_1, \lambda_2$ of $\hat \gamma$ with common end points $x,x'$, where $\partial \lambda_1 = \partial \lambda_2 = \{x,x'\}$ are in different components of $\mathbb S$ and int $\gamma_i$ is disjoint from $\mathbb S$ for $i=1,2$. We can also assume without loss of generality that the double point component of $\hat \gamma$ containing $x$, is part of the limit of the double sets for $\gamma_i$. We will prove that this contradicts the assumption that $\hat \gamma$ achieves the minimum of length amongst representatives of $\gamma$ if there is a non-zero angle between $\lambda_1, \lambda_2$ at $x$. See Fig. \ref{fig:0123}b.
 
 We can choose two subarcs of $\gamma_i$ starting at the sequence of double points $x_i$  which converge to $x$ as $i \to \infty$, with the other ends $x'_i, x''_i$ both converging to $x'$. These arcs converge to the two branches $\lambda_1, \lambda_2$ of $\hat \gamma$. Moreover the tangents to $\lambda_1, \lambda_2$ agree at $x'$, since clearly this is a non-crossing double point of $\hat \gamma$. See Fig. \ref{fig:0123}c. The key idea is to use the Meeks-Yau exchange and round-off trick (\cite{meeks}) to get a contradiction to the assumption that the length of $\hat \gamma$ is the infimum of lengths in the class of $\gamma$. This will not only eliminate the extra self-intersection of $\hat \gamma$ but also the problem of the complement of the disks $\hat D_i$ failing to be contractible. 

 We modify $\hat \gamma$ in two stages. See Fig. \ref{fig:0123}d. Firstly we switch the arcs $\lambda_1$ and $\lambda_2$. Note that the new loop is smooth at $x'$ but not at $x$.  Next we need to `smooth the corners' at $x$ created by this exchange. The idea is to use two unit radius circles tangent to two branches of $\hat \gamma$ at $x$ as in Fig. \ref{fig:45} top.  We then replace parts of these branches with an arc $\mu, \nu$ of each circle to produce a new $C^1$ path, which is the required further modification of $\hat \gamma$. Since radial projection onto $\mu, \nu$ reduces length, replacement of segments of the two branches of $\hat \gamma$ by $\mu, \nu$ decreases the length of $\hat \gamma$. Let us denote the further modified loop by $\gamma^*$.
 

\begin{figure} [h!]
\centering
\includegraphics[width=.6\textwidth]{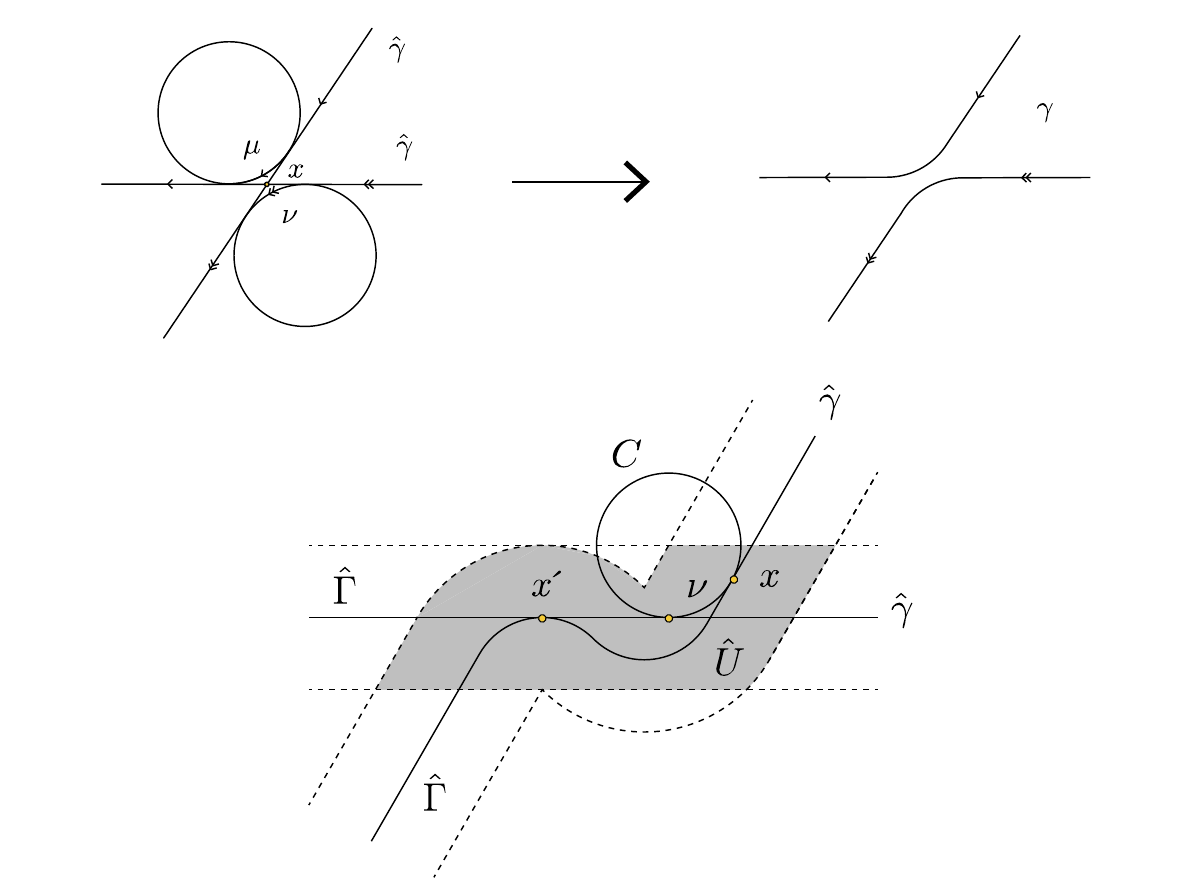}
\caption{Top: Smoothing the corners at $x$. Bottom: An example where the centres of the unit circles are where the boundary of the ribbon $\hat \Gamma$ crosses itself.}
\label{fig:45}
\end{figure}

Next we claim that no new double points are introduced in $\gamma^*$ as compared with $\hat \gamma$. It suffices to show that there are no intersections of $\hat \gamma$ and $\mu$, except at the endpoints of $\mu$. If we can show that $\mu, \nu$ are contained in $\hat U$ then this follows. But the centres of the unit radius circles are where the boundary of the ribbon $\hat \Gamma$ crosses itself - see Fig. \ref{fig:45} bottom. Hence it is easy to see that as well as $\mu, \nu \subset \hat U$, the ribbon $\Gamma^*$ for $\gamma^*$ is contained in $\hat \Gamma$. Moreover, the modified loop $\gamma^*$ clearly satisfies that the modified disks $D^*_1, D^*_2$ have contractible complements - see Fig. \ref{fig:6}. 


\begin{figure} [h!]
\centering
\includegraphics[width=1\textwidth]{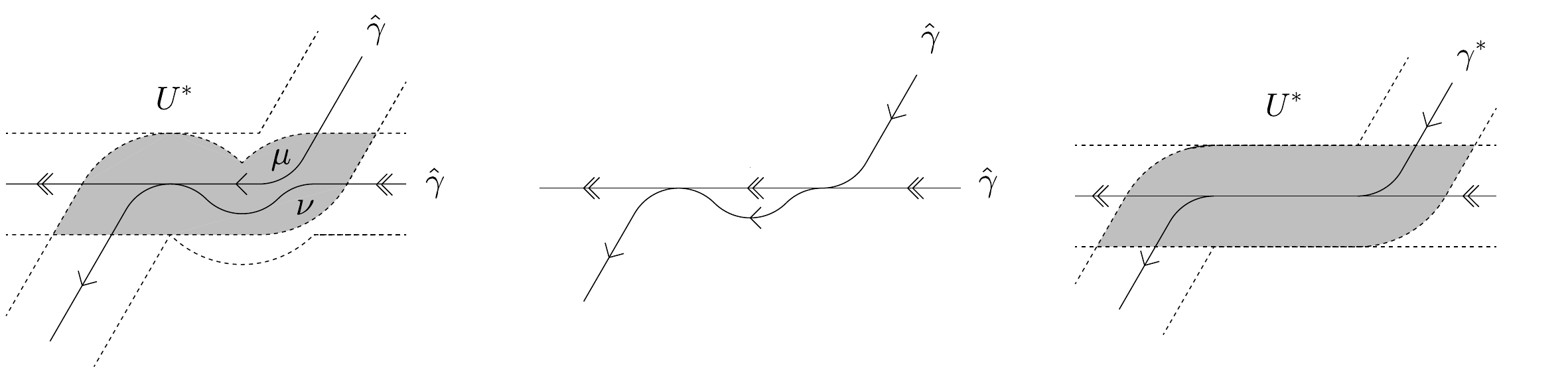}
\caption{A modification of $\hat \gamma$ such that the resultant loop $\gamma^*$ has disks $D^*_1, D^*_2$ of contractible complements.}
\label{fig:6}
\end{figure}

For the final crossing condition for $\gamma^*$, we need to show that $D^*_j \cap (S^1 \times \{0\})$ is connected for $j=1,2$. We know that $D^i_j \cap (S^1 \times \{0\})$ is connected for $j=1,2$ and for all $i$. Moreover $\hat D_j$ is the limit of $D^i_j$ as $i \to \infty$ and $\hat \gamma$ is the limit of $\gamma_i$ as $i \to \infty$. It is easy to see that the modification of $\hat \gamma$ to $\gamma^*$ gives  $\hat D_j \cap (S^1 \times \{0\}) = D^*_j \cap (S^1 \times \{0\})$ and so we see immediately that this is also connected. 
So this completes the proof that $\hat \gamma$ or $\gamma^*$ satisfies this crossing condition.

Hence we have shown that the assumptions that the crossing condition is a single double point or interval for $\hat \gamma$ in $\hat U$ is violated and the angle at $x$ between the branches of $\hat \gamma$ is non-zero lead to a contradiction. 

It remains to consider the case that the angle at $x$ between the branches of $\hat \gamma$ is zero. In this case, suppose without loss of generality that $\|\lambda_1|\le \|\lambda_2\|$. We modify $\hat \gamma$ by replacing $\lambda_2$ with a copy of $\lambda_1$. Hence for the modified loop, which we denote again by $\gamma^*$, the double set now contains $\lambda_1$ and the length is at most that of $\hat \gamma$. We can continue doing this if there are further such pairs of branches of $\gamma^*$ with ends which are tangential, until no such branches occur. This clearly gives a new loop in the same class as $\gamma$ with minimum length. Moreover for this modified loop, we see that $\hat D_j \cap (S^1 \times \{0\}) = D^*_j \cap (S^1 \times \{0\})$ is connected and the complement of $D_j$ is contractible - see Fig. \ref{fig:6}.

So this completes the proof of all parts of the crossing condition for $\gamma^*$ and the argument also shows that it is in the same diagram class as $\gamma$.
 
\end{proof}

\begin{definition} The ribbonlength $\mbox{Rib}(\gamma)$ of a ribbon knot or link $\gamma$ is the ratio of the length of the core $\gamma$ to the width of the ribbon $\Gamma$,
$$\mbox{Rib}(\gamma)= \frac{\mbox{Length}(\gamma)}{\mbox{Width}(\Gamma)}= \frac{\mbox{Length}(\gamma)}{2}$$
\end{definition}

\section{Length minimisers in disk space}

\begin{theorem}\label{ribcs} \begin{enumerate}
\item Length minimisers in disk space exist and are $cs$. 
\item Suppose that a length minimiser of a ribbon knot or link in disk space is ribbon. Then this disk space minimiser is also a length minimiser in ribbon space. 
\end{enumerate}
\end{theorem}

\begin{proof}
By Proposition \ref{disk}, Observation \ref{link} and Theorem \ref{doublepoint} all the ribbon loop diagrams or link diagrams in a given class are contained in a component of disk space, which we will denote by $\mathcal D_\gamma$. Here we abuse notation by using $\gamma$ as both a single realisation of a given ribbon loop or link diagram or the class of such realisations. 

Firstly, we need to prove that $\mathcal D_\gamma$ has length minimisers. Later on we will show that such minimisers are of $cs$ type. Firstly in disk space, loops satisfy the weak separation bound, so by Observation \ref{obs} and the Arzela-Ascoli theorem, we can find a $C^1$ limit $\hat \gamma$ of a sequence of loops or links $\gamma_i$ in  $\mathcal D_\gamma$ whose lengths converge to the infimum of length. To show that $\hat \gamma$ is in  $\mathcal D_\gamma$, we need to first establish that $\hat \gamma$ satisfies the weak separation condition. This follows by the approach in Theorem \ref{minimalribbon}. In particular at a sequence of points $x_i$ on $\gamma_i$ converging to $x$ on $\hat \gamma$, there are two circles of radius one which touch $\gamma_i$ at $x_i$ on either side. These circles converge to a similar pair of circles touching $\hat \gamma$ at $x$ on each side, establishing the weak separation condition. 

  Note that the double points of $\hat \gamma$ don't necessarily exactly match those of $\gamma_i$ under taking limits. There are two ambiguities in the matchings. The first is that isolated double points may switch with double intervals. The second is that new non-crossing double points or double intervals may appear in the limit $\hat \gamma$. However it is still true that the complementary regions of $\hat \gamma$ match those of members of $\gamma_i$ under limits and there are still unit radius disks in these regions for $\hat \gamma$, by taking limits of such disks in the corresponding regions of $\gamma_i$.  (Recall that non-crossing double points and double intervals do not separate complementary regions, so having new such points and intervals in $\hat \gamma$ does not change the class of loop or link). Consequently, this completes the proof that minimisers of length exist in $\mathcal D_\gamma$.

Clearly it suffices to show that $C^1$ minimisers of length in $\mathcal D_\gamma$ are always of $cs$ type, to complete the first part of the Theorem. 


\begin{figure} [h!]
\centering
\includegraphics[width=1\textwidth]{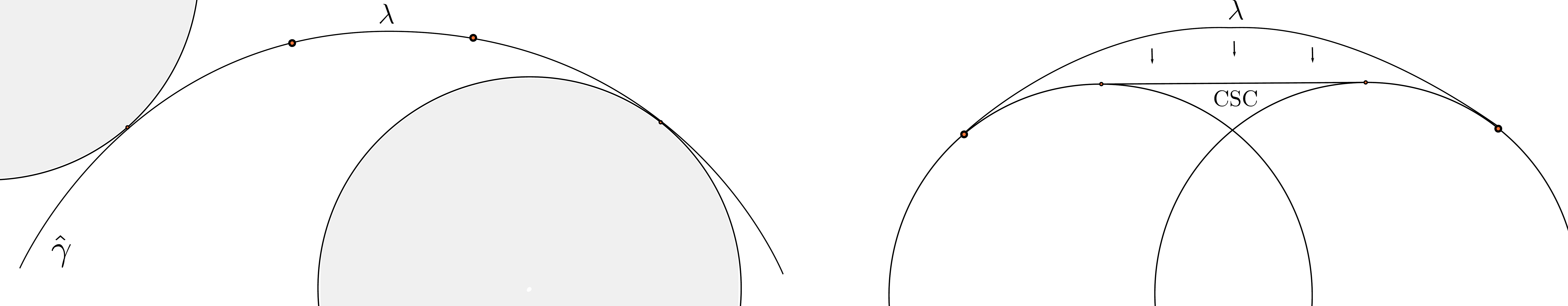}
\caption{Left: The subarc $\lambda$ does not touch the unit radius disks (in grey) in the complementary regions for $\hat \gamma$ on either side of $\lambda$. Right: We shorten $\lambda$ by replacing it by a Dubins solution of type {\sc{csc}} (circle-line segment-circle). }
\label{fig:7}
\end{figure}

The argument that a $C^1$ length minimiser $\hat \gamma$ of $\mathcal D_\gamma$ is of $cs$ type is by contradiction. We can divide $\hat \gamma$ into a finite collection of short subarcs with the following properties. Suppose first one of the short subarcs $\lambda$  of $\hat \gamma$ does not touch the unit radius disks in the complementary regions for $\hat \gamma$ on either side of $\lambda$. We can then shorten $\lambda$ by replacing it by a Dubins solution (see \cite{ayala} and Fig. \ref{fig:7}) with the same endpoints and end directions as $\lambda$ if the latter is not of $cs$ type. We can choose the subarcs so that any such a $\lambda$ is short enough relative to its distance from the adjacent unit radius disks, so that the Dubins solution will also miss these disks. 

We also need to consider subarcs $\lambda$ of $\hat \gamma$ with one or both ends on a unit radius disk in a complementary region, so that the interior of $\lambda$ is disjoint from any of these disks. Again if $\lambda$ is short enough and not of $cs$ type, then we can replace it by a Dubins solution with the same endpoints and end directions as $\lambda$ so that the new arc has still has interior disjoint from these disks. See Fig. \ref{fig:7} and \cite{ayala}. Note that the unit disks at the ends of $\lambda$ are tangential to $\lambda$.

Hence we have found a shorter loop or link in  $\mathcal D_\gamma$, contradicting our choice of $\hat \gamma$ as shortest. 

Next, by the above argument it follows that $\hat \gamma$ is a collection of Dubins solutions possibly connecting points on the boundaries of the unit disks in the complementary regions. Moreover at these unit disk boundary points, $\hat \gamma$ must be tangent to the unit circles. To conclude that $\hat \gamma$ is $cs$,  it remains to verify it is a finite concatenation of Dubins curves together with arcs of unit disks joined smoothly together. If there were infinitely many Dubins curves, then these would end at infinitely many points on the boundary of the unit disks. Such points would accumulate. But if two such points are very close, it is easy to see \cite{ayala} that any embedded Dubins curve joining them is the arc of the unit circle. Hence at an accumulation point, we would have such arcs joining up smoothly and hence do get a finite number of Dubins curves as required, unless infinitely many of these arcs had self intersections. But in the latter case, $\hat \gamma$ would have infinitely many self-crossing points, which is a contradiction.  So this establishes a ribbon length minimiser is of $cs$ type and is in $\mathcal D_\gamma$. 

By assumption, if the ribbon conditions of separation and crossing are satisfied by this disk length minimiser, then this provides a ribbon length minimiser, since disk space contains ribbon space. This completes the proof.  
\end{proof}


\section{Variational approach}

In this section we study critical points of length for elements in disk space $\mathcal D$. Then in Section \ref{ribexamples} we will show how the ribbonlength varies as the knot diagram changes, so long as the length minimiser in disk space is a ribbon loop. 

A useful physical model for studying critical points of length for loops in $\mathcal D$ is a finite collection of $n$ unit radius disks (each of these corresponds to a bounded complementary region of the loop) and an elastic band being the loop itself. For each of the $n$ disks, there are two degrees of freedom, which are the coordinates of its centre. Accordingly, the configuration space is a subset of $\mathbb{R}^{2n}$. In addition, there is a distance constraint, that the centres of a pair of unit disks must be separated by a distance of at least 2 (see also Proposition \ref{disk}). Our approach is based on the following proposition from the work of Kirszenblat et. al. on minimal curvature-constrained networks \cite{kirszenblat} inspired by the work of Rubinstein and Thomas on Steiner trees \cite{rubthom}.

\begin{figure} [h!]
\centering
\includegraphics[width=1\textwidth]{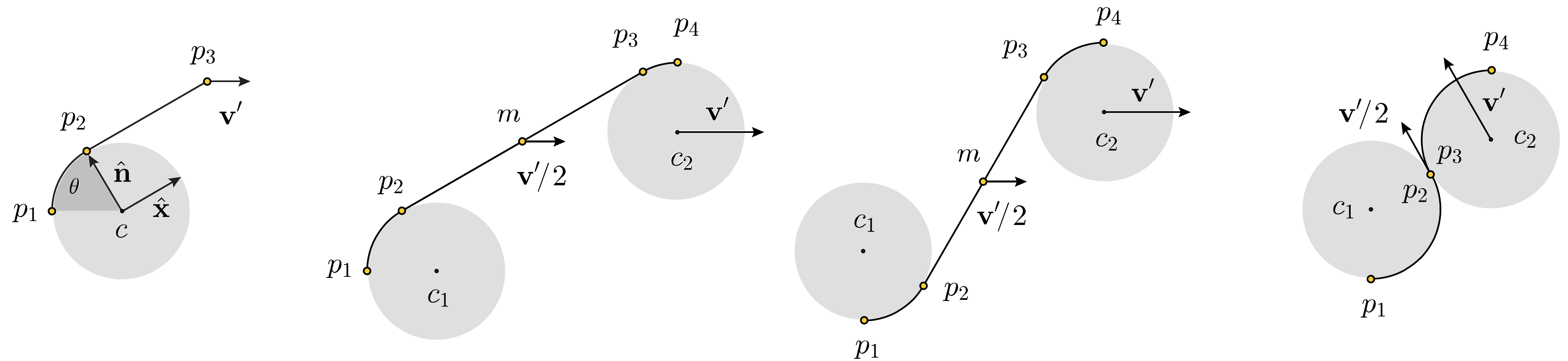}
\caption{Left: an illustration of Proposition \ref{proposition41}. The other illustrations are for Proposition \ref{proposition42}.}
\label{fig:Variation}
\end{figure}

\begin{proposition}
	Suppose that $p_1 p_2 p_3$ is a smooth curve consisting of a concatenation of an arc of a unit circle and a line segment with common point $p_2$ such that  $p_1$ is held fixed while $p_3$ is varied, as shown in Fig. \ref{fig:Variation} left. Then the first variation of length of $p_1 p_2 p_3$ is the inner product between the direction of variation and the unit vector from $p_2$ to $ p_3$.
\label{proposition41}
\end{proposition}
	\begin{proof} Suppose that $p_3$ moves along any smooth curve with derivative at its initial position being the vector $\mbox{\bf v}$. Let $\ell$ denote the length of $p_1 p_2 p_3$. The curve $p_1 p_2$, of length $\theta$, is an arc of a unit circle whose centre $c$ is considered the origin. 
	
	Let $\hat{\mbox{\bf n}}$ be the unit vector from the origin to the point $p_2$ and ${\mbox{\bf x}}$
the vector from $p_2$ to $p_3$, with length denoted by $x$ and unit vector in the direction of ${\mbox{\bf x}}$ denoted by $\hat{\mbox{\bf x}}$. Observe that the first variation of the vector ${\mbox{\bf x}}$ is equal to the first variation of its head $\mbox{\bf v}$ minus the first variation of its tail $\hat{\mbox{\bf n}}$. If the point $p_3$ is perturbed in the direction of $\mbox{\bf v}$, then the first variation of length of $p_1 p_2 p_3$ is
		\begin{align*}
		\ell' &= \theta' + x'\\
		&=  \hat{\mbox{\bf n}}' \cdot \hat{\mbox{\bf x}} + {\mbox{\bf x}}' \cdot \hat{\mbox{\bf x}}\\
		&=  \hat{\mbox{\bf n}}' \cdot \hat{\mbox{\bf x}} + (\mbox{\bf v}-\hat{\mbox{\bf n}}' ) \cdot \hat{\mbox{\bf x}} \\
                  &= \mbox{\bf v} \cdot \hat{\mbox{\bf x}}
		\end{align*} 
\end{proof}
As a consequence of Proposition \ref{proposition41}, we obtain the following proposition.
\begin{proposition}
Suppose that $p_1p_2p_3p_4$ is a smooth curve consisting of a concatenation of an arc of a unit radius circle with centre $c_1$ (fixed), followed by a line segment, followed by an arc of a unit radius circle with centre $c_2$ (variable), as shown in Fig. \ref{fig:Variation}. Then the first variation of length of $p_1p_2p_3p_4$ is the scalar product between the vector of variation and the unit vector from $p_2$ to $ p_3$.
\label{proposition42}
\end{proposition}

\begin{proof}
Suppose that the centre $c_2$ (respectively the point $p_4$) moves along any smooth curve with derivative at its initial position being the vector $\mbox{\bf v}$ (respectively a parallel translation of this curve with the derivative $\mbox{\bf v}$ at $p_4$). Let $\ell$ denote the length of the curve $p_1p_2p_3p_4$.  We may view the case where the line segment has zero length as a limiting case, see Fig. \ref{fig:Variation} right. In this case, the points $p_2$ and $p_3$ coincide and we define the unit vector $\hat{\mbox{\bf x}}$ to be tangent to the disks at the point $p_2$ and pointing in the direction of motion along the curve when traveling from the point $p_1$ toward the point $p_4$. Let $m$ denote the midpoint of the straight segment $p_2p_3$, which coincides with the points $p_2$ and $p_3$ if the straight segment is degenerate. Let $\ell_1$ and $\ell_2$ denote the respective lengths of the curves $p_1p_2m$ and $mp_3p_4$, both of which are curves as in Proposition \ref{proposition41}. Observe that the first variation of the midpoint $m$ is half that of the centre $c_2$. By proposition \ref{proposition41}, if the point $m$ is perturbed in the direction of $\mbox{\bf v}/2$, then the first variation of length of  $p_1p_2m$ is
\begin{eqnarray*}
\ell_1' = \frac{1}{2} \mbox{\bf v} \cdot \hat{\mbox{\bf x}}
\end{eqnarray*}
On the other hand, if we subtract $\mbox{\bf v}$ from all elements in Fig \ref{fig:Variation}, then we see that the first variation of length of $mp_3p_4$ is
\begin{eqnarray*}
\ell_2' &= -\frac{1}{2} \mbox{\bf v} \cdot - \hat{\mbox{\bf x}}\\
& = \frac{1}{2} \mbox{\bf v} \cdot \hat{\mbox{\bf x}}
\end{eqnarray*}
Summing these two terms, we obtain for the first variation of length of $p_1p_2p_3p_4$
\begin{equation*}
\ell' = \ell_1'+\ell_2' = \mbox{\bf v} \cdot \hat{\mbox{\bf x}}
\end{equation*}
\end{proof}

\begin{observation}We have established that minimal length elements in disk space $\mathcal D$ must be $cs$ disk diagrams, in Theorem \ref{ribcs} . To compute the first variation of the length of a $cs$ disk diagram (as the positions of the disks are varied) we first decompose the $cs$ disk diagram into curves, each of which is a concatenation of an arc of a unit radius circle, followed by a line segment, followed by an arc of a unit radius circle as in Proposition \ref{proposition42}. Then, we simply add the contributions coming from each piece in this decomposition. Concretely, suppose that the $i$th disk moves along a smooth curve $t \to {\bf c}(t) \in \mathbb{R}^2$ with velocity vector ${\bf v} = \frac{d{\bf c}}{dt}$ while the remaining disks are held fixed. Moreover, suppose that there are $2k$ (possibly degenerate) straight segments lying tangent to the $i$th disk with outward pointing unit vectors $\hat{\bf u}_1, \hat{\bf u}_2, \ldots, \hat{\bf u}_{2k}$, see Fig. \ref{fig:perturbation}. Let $\ell$ denote the length of the loop $\gamma$ and $(x_{2i - 1}, x_{2i})$ the centre of the $i$th disk. Then the directional derivative $D\ell({\bf v})$ of $\ell$ in the direction of ${\bf v}$ at the configuration point ${\bf x} = (x_1, x_2, \ldots, x_{2n})$ is given by
\begin{align*}
D\ell({\bf v}) &= \lim_{t \to 0^+}\ell(x_1, x_2, \ldots, x_{2i - 1} + c_1(t), x_{2i} + c_2(t), \ldots, x_{2n}) - \ell(x_1, x_2, \ldots, x_{2n})\\
& = (\hat{\bf u}_1+ \hat{\bf u}_2+ \ldots +\hat{\bf u}_{2k}) \cdot {\bf v}.
\end{align*}
\end{observation}
\begin{figure} [h!]
\centering
\includegraphics[width=.3\textwidth]{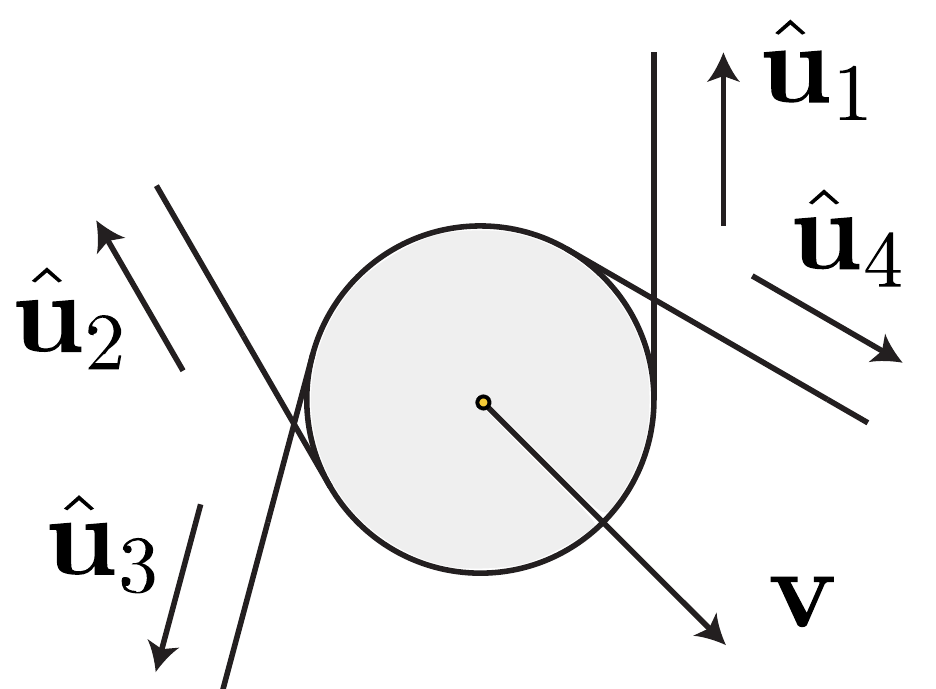}
\caption{}
\label{fig:perturbation}
\end{figure}


\section{Examples of ribbonlength minimisers}\label{ribexamples}

We compute the minimal ribbonlength for some small knot and link diagrams. We discuss the minimal ribbonlength of the standard figure-8 knot diagram and observe that there exists a minimum length disk diagram that does not satisfy the separation bound. We give a conjecture about the minimal ribbonlength for the standard figure-8 knot diagram. We finish this section by providing an infinite family of link diagrams whose minimal ribbonlength can be easily computed.


\subsection{The minimal ribbonlength of the unknot}\hfill

Since $\mbox{Length}(\gamma)=2\pi$, the ribbonlength of the unknot is clearly $\mbox{Rib(unknot)}=\pi$ for the standard diagram $\gamma$, see Fig. \ref{fig:diagramunknot}. We will see in the next section that this is also the minimal ribbon length over all diagrams of the unknot. 


\begin{figure} [h!]
\centering
\includegraphics[width=.9\textwidth]{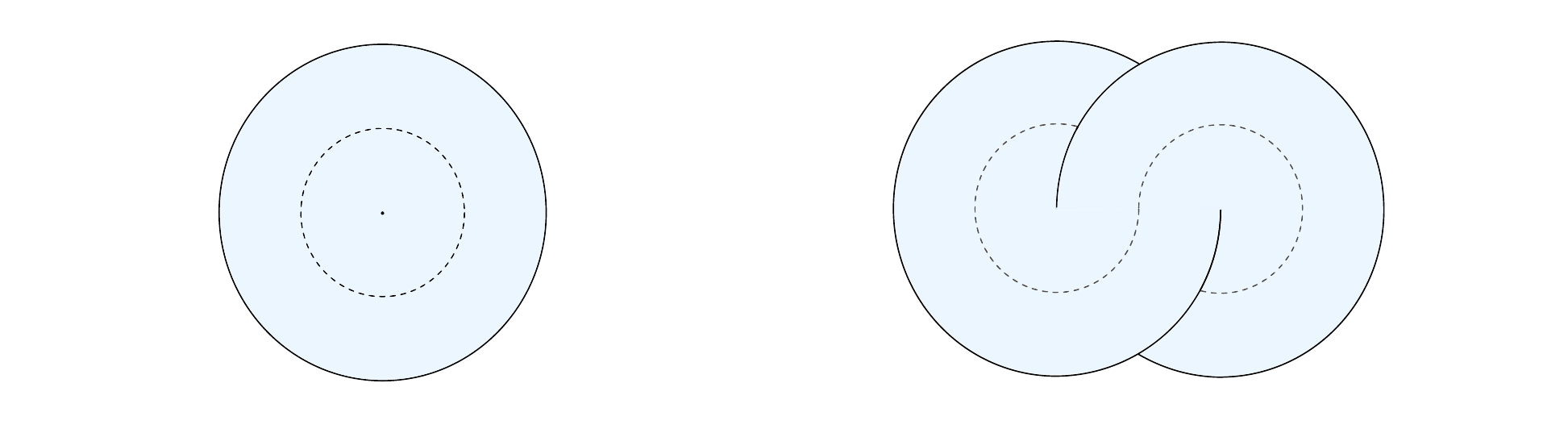}
\caption{Left: the ribbon knot of minimal ribbonlength for the standard unknot diagram. Right: the ribbon knot of minimal ribbonlength for the standard unknot with one twist. }
\label{fig:diagramunknot}
\end{figure}

\subsubsection{The twisted unknot.}\label{twisted} We compute the effect of a twist on the ribbonlength of the unknot. See Fig. \ref{fig:diagramunknot}. Without loss of generality, consider the disk diagram in Fig. \ref{fig:diskdiagrams} $(a)$, and fix the position of $D_0$. We claim that $D_1$ must touch $D_0$ at a single point. Suppose to the contrary that $D_1$ does not touch $D_0$. Let the contribution to the negative gradient of length from the centre of $D_1$ be denoted by $\bf{u}$. If $D_0$ and $D_1$ do not touch, we can always decrease the length of the disk diagram by perturbing the disk $D_1$ in the direction of the negative gradient $\bf{u}$. In addition, since $D_0$ and $D_1$ belong to different regions separated by the disk diagram, the distance between their centres is at least 2. We conclude that $D_0$ and $D_1$ intersect at a single point. It is easy to see that the minimal length disk diagram is in fact a minimal length ribbon loop. Since the length of the minimal ribbon loop is $4\pi$, the ribbonlength of a twisted unknot is $2\pi$. 

In general, a twist adds $2 \pi$ to the length of a minimal disk diagram by a similar argument. 


\subsection{The minimal ribbonlength of the standard trefoil knot diagram}\label{trefoil}\hfill

We will prove that the disk diagram for the standard trefoil diagram in Fig. \ref{fig:DiagramTrefoil} left is a length minimising ribbon loop. The disks are labelled $D_0$, $D_1$, $D_2$ and $D_3$ as in Fig. \ref{fig:diskdiagrams} $(b)$-$(d)$. Let us fix the position of the disk $D_0$. We will show that each of the three disks $D_1$, $D_2$ and $D_3$ touches the disk $D_0$ and a length minimising loop is connected to it by straight segments. Suppose to the contrary that one of the three disks $D_1$, $D_2$ and $D_3$, say $D_2$, touches no other disks. In order for the disk $D_2$ to be in mechanical equilibrium, meaning that we cannot decrease the length of the curve by perturbing the disk $D_2$, the disk diagram must be as in Fig. \ref{fig:DiagramTrefoil} $(b)$. In this case, the disk $D_2$ is connected to the disk $D_0$ by two straight segments and to each of the disks $D_1$ and $D_3$ by one straight segment. We can decrease the length of the curve by perturbing one of the two disks $D_1$ and $D_3$, say $D_1$, in the direction of a vector ${\bf v}$ which makes an acute angle with the negative gradient.


\begin{figure} [h!]
\centering
\includegraphics[width=.8\textwidth]{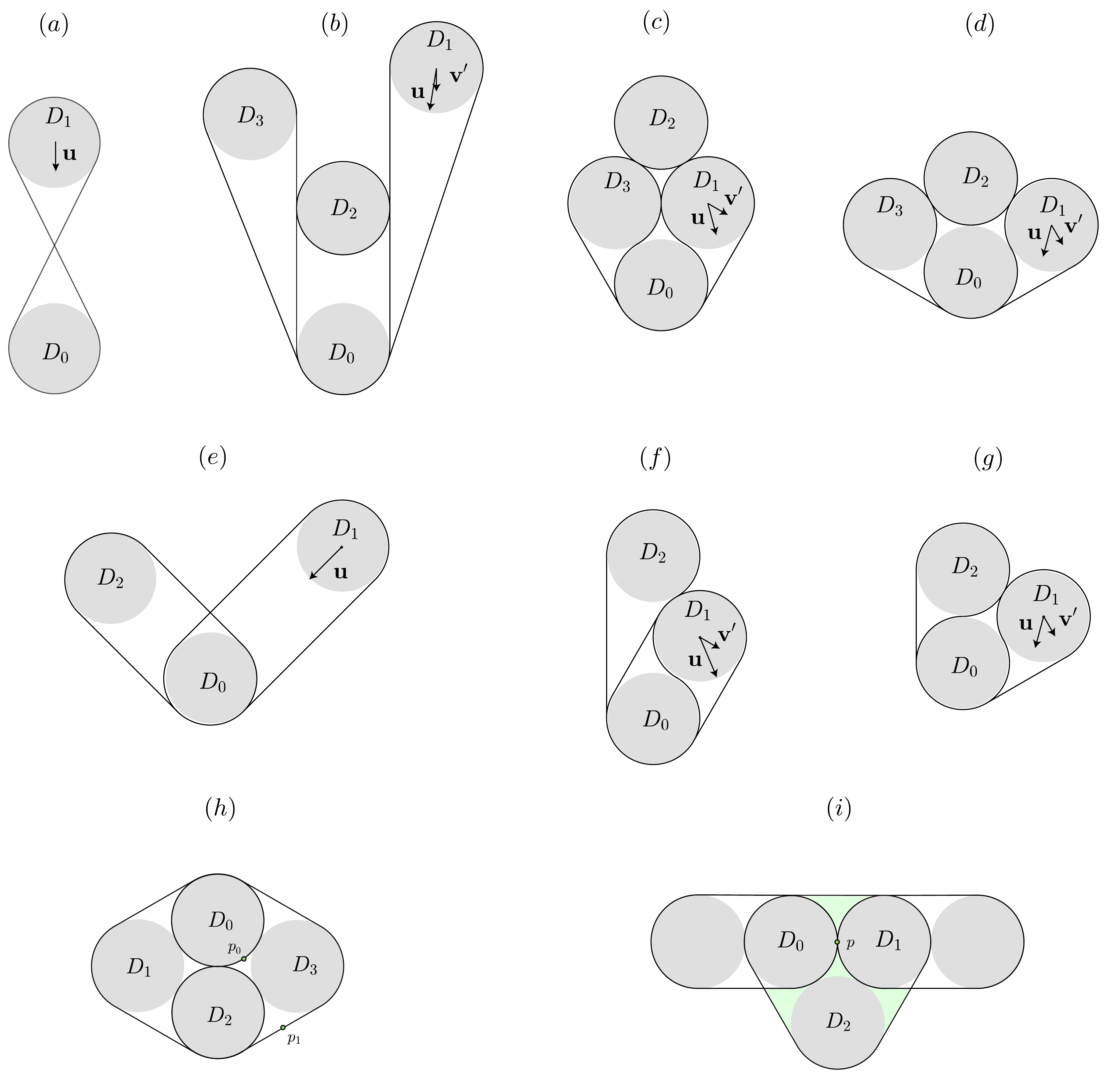}
\caption{The black trace represents a disk diagram for the standard trefoil knot diagram for $(b)$-$(d)$ and for the standard Hopf link diagram for $(e)$-$(g)$. The vector $\bf{u}$ denotes a contribution to the negative gradient for the length of a disk diagram. The centre of $D_1$ moves along a smooth curve with derivative at its initial position being the vector $\mbox{\bf v}$. The illustration $(h)$ is a minimal crossing trefoil disk diagram that violates the separation bound as $||p_0-p_1||<2$. In $(i)$ the green portions are part of the same region as $p$ is a tangential (fake) crossing. The disk $D_2$ is obtained from Proposition \ref{disk}, see Observation \ref{fake}.}
\label{fig:diskdiagrams}
\end{figure}

Next, suppose that one of the three disks $D_1$, $D_2$ and $D_3$, say $D_2$, is prevented from touching the disk $D_0$ by at least one of the two disks $D_1$ and $D_3$, see Fig. \ref{fig:diskdiagrams} $(c)$. In this case, we cannot perturb the disk $D_1$ in the direction of the negative gradient, because it will intersect the disk $D_0$. However, we can perturb the disk $D_1$ in the direction of a vector $\bf{v}$ which makes an acute angle with the negative gradient in order to decrease the length of the curve.

Finally, suppose that each of the disks $D_1$, $D_2$ and $D_3$ touches the disk $D_0$, but that at least one of the disks $D_1$, $D_2$ and $D_3$, say $D_2$, is not connected to the disk $D_0$ by straight segments, see Fig. \ref{fig:diskdiagrams} $(d)$. In this case, we can perturb one of the disks $D_1$ or $D_3$ in the direction of a vector $\bf{v}$ which makes an acute angle with the negative gradient in order to decrease the length of the curve.

 Observe that there is a 2-parameter family of length minimisers that is obtained by rotating two of the three disks $D_1, D_2$ and $D_3$ about the disk $D_0$. To see this, observe that the contribution to the negative gradient from the position of, say, the disk $D_1$, which we denote by ${\bf u}$, points from the centre of the disk $D_1$ to that of the disk $D_0$. Since the disk $D_0$ is held fixed, the smallest admissible angle between the vectors ${\bf u}$ and ${\bf v}$, i.e. the perturbation vector, is $\pi/2$. That is, the rate of change of the length of the core $\gamma$ is at least zero. This is true so long as we rotate the disk $D_1$ about the disk $D_0$ in some neighbourhood of the original configuration.

In Fig. \ref{fig:diskdiagrams} $(h)$ we illustrate the other minimal crossing disk diagram for the trefoil. This cannot be a ribbon loop since it violates, for example, the separation bound since $||p_0-p_1||< 2$.

Since $\mbox{Length}(\gamma)= 12 + 4\pi$, the ribbonlength of the standard trefoil knot diagram is $\mbox{Rib}(\mbox{trefoil})= 6 + 2 \pi$.


\begin{figure} [h!]
\centering
\includegraphics[width=\textwidth]{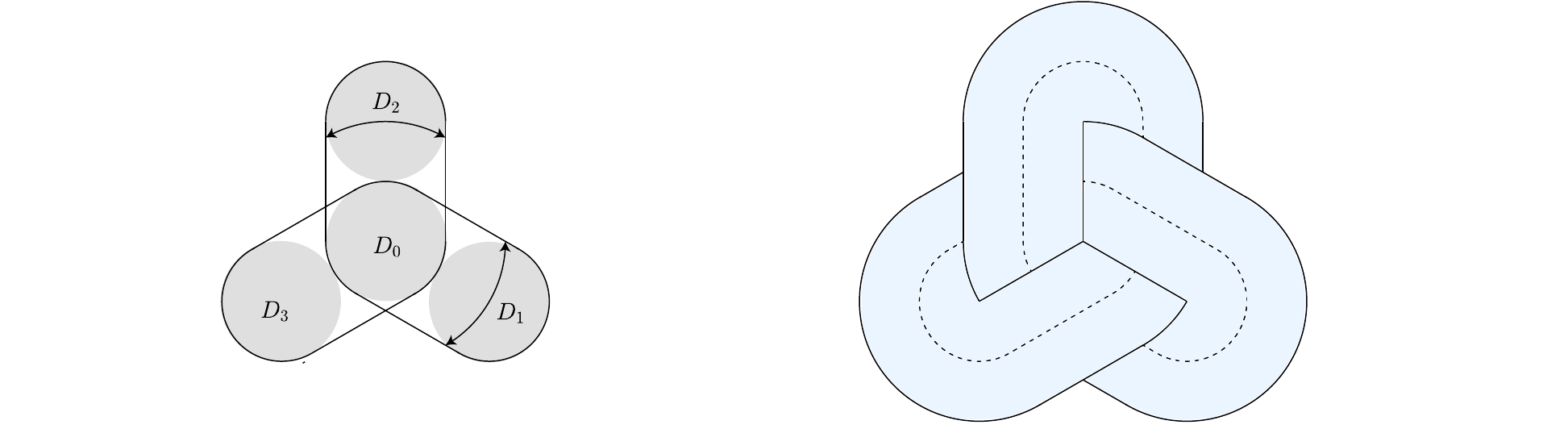}
\caption{Left: A length minimising ribbon loop (core) for the standard trefoil diagram. Right: a ribbon knot of minimal ribbonlength for this trefoil diagram.}
\label{fig:DiagramTrefoil}
\end{figure}


\begin{observation}{\bf Tangential crossings do not separate disk regions}. \label{fake} Recall that in Definition \ref{diskspace} we allowed crossings to be tangential in disk space. Consider three simple loops with an intersection pattern like that of three Olympic circles (a planar projection of the Borromean rings). Note that a minimal double point representative of this  diagram separates the plane into six regions, one of which is unbounded. A tight disk diagram for the three Olympic circles is shown in Fig. \ref{fig:diskdiagrams} $(i)$. Therefore the length minimiser in disk space divides the plane into seven connected domains, two of which are separated by a tangential crossing denoted by $p$. In Fig. \ref{fig:diskdiagrams} $(i)$ the green portions are part of the same complementary region and $D_2$ is the disk obtained from Proposition \ref{disk}.
\end{observation} 


\subsection{The minimal ribbonlength for the standard Hopf link diagram}\label{hopf}\hfill 

We will prove that the disk diagram corresponding to the standard Hopf link diagram in Fig. \ref{fig:DiagramHopfLink} left is a length minimising ribbon loop. The disks are labelled $D_0$, $D_1$ and $D_2$ as in Fig \ref{fig:diskdiagrams} $(e)$-$(g)$. Let us fix the position of the disk $D_0$. We will show that each of the two disks $D_1$ and $D_2$ touches the disk $D_0$ and is connected to it by straight segments for a length minimising representative. Suppose to the contrary that at least one of the two disks $D_1$ and $D_2$, say $D_1$, does not touch the disk $D_0$. In particular, the disk $D_1$ touches no other disks. In Fig. \ref{fig:diskdiagrams} $(e)$, the contribution to the negative gradient from the position of the disk $D_1$ is denoted by $\bf{u}$. We can decrease the length of the curve by perturbing the disk $D_1$ in the direction of the negative gradient. Similarly, if the two disks $D_1$ and $D_2$ touch one another but do not touch the disk $D_0$, then we can perturb the two disks $D_1$ and $D_2$ collectively in the direction of the negative gradient.

Next, suppose that one of the two disks $D_1$ and $D_2$, say $D_1$, serves as an obstacle preventing the disk $D_2$ from touching the disk $D_0$, see Fig.\ref{fig:diskdiagrams} $(f)$. Let $\bf{u}$ denote the contribution to the negative gradient from the position of the disk $D_1$. In this case, we cannot perturb the disk $D_1$ in the direction of the negative gradient, because it touches the disk $D_0$. However, we can perturb the disk $D_1$ in the direction of a vector $\bf{v}$ which makes an acute angle with the vector $\bf{u}$. In this way, we can decrease the length of the curve.

Finally, suppose that both of the disks $D_1$ and $D_2$ touch the disk $D_0$, but that neither of the two disks $D_1$ and $D_2$ is connected to the disk $D_0$ by straight segments, see Fig. \ref{fig:diskdiagrams} $(g)$. In this case, we can perturb the disk $D_1$ in the direction of a vector $\bf{v}$ which makes an acute angle with the negative gradient in order to decrease the length of the curve.

Since $\mbox{Length}(\gamma)=8 + 4\pi$, the ribbonlength of the standard Hopf link diagram is $\mbox{Rib}(\mbox{Hopf link})=4+2\pi$.


\begin{figure} [h!]
\centering
\includegraphics[width=1\textwidth]{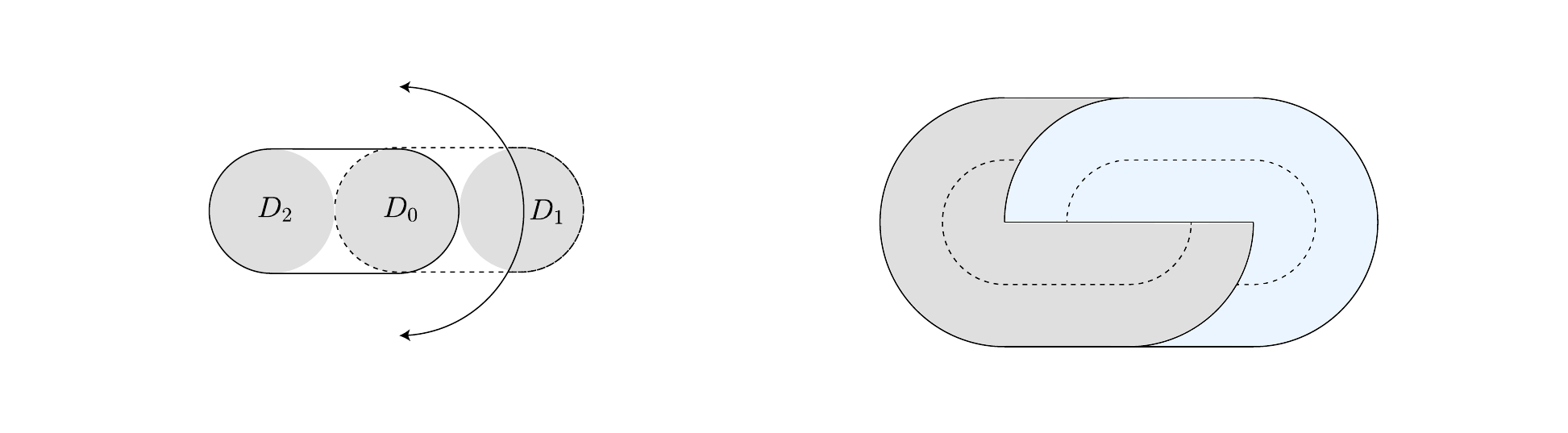}
\caption{Left: A length minimising ribbon loop (core) for the standard Hopf link diagram. Note that there is a 1-parameter family of minimisers (up to isometries) that is obtained by rotating $D_1$ or $D_2$ about the disk $D_0$. Right: the ribbon link of minimal ribbonlength for the standard Hopf link diagram.}
\label{fig:DiagramHopfLink}
\end{figure}



\subsection{On the ribbonlength of the standard figure-8 knot diagram}\label{fig8} \hfill

 Regarding the disk diagram for the standard figure-8 knot diagram, there is a simple argument to show that the configuration depicted in Fig. \ref{fig:dcp} left, where the centres of the disks $D_0, D_1, D_2, D_3$ and $D_4$ are positioned at the points $(0, 0), (-\sqrt{3}, -1), (0, -2), (\sqrt{3}, -1)$ and $(0, 2)$, respectively, gives a minimum length element in disk space $\mathcal D$. To see this, observe that we can decompose the core $\gamma$ into three loops $\gamma_1, \gamma_2$ and $\gamma_3$ as in Fig. \ref{fig:dcp}. By a variational argument similar to those given above, it is clear that the respective lengths of the loops $\gamma_1, \gamma_2$ and $\gamma_3$ must be at least $2 \pi, 4 + 2 \pi$ and $8 + \pi$ due to the number of disks contained inside each of the three loops. Since the configuration depicted in Fig. \ref{fig:dcp} left attains all three bounds, it gives a minimum length element in disk space $\mathcal D$ for disk diagrams of the combinatorial type of the standard figure-8 knot diagram. The corresponding ribbon is illustrated in Fig. \ref{fig:knotfig-8}. Note there is a 1-parameter family of length minimisers that is obtained by rotating the disk $D_4$ about the disk $D_0$.


\begin{figure} [h!]
\centering
\includegraphics[width=.9\textwidth]{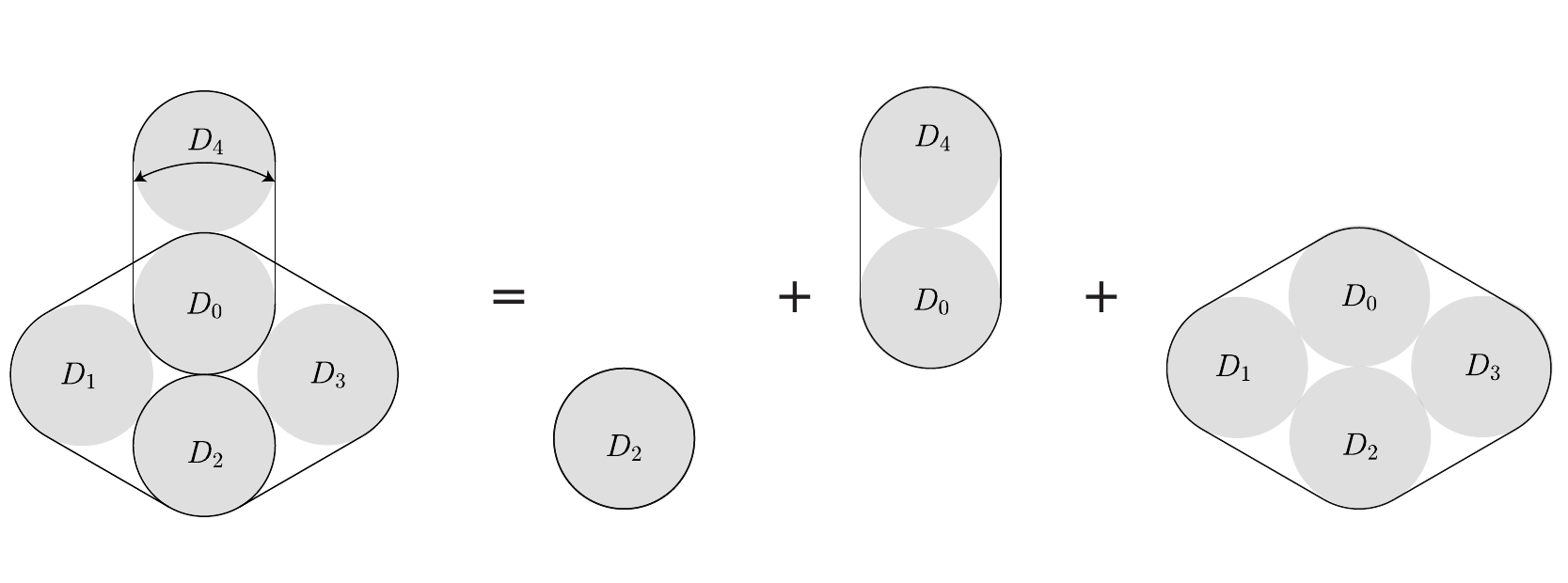}
\caption{A decomposition of a minimal disk diagram for the standard figure-8 knot diagram.}
\label{fig:dcp}
\end{figure}

\begin{observation} The previous example of a minimiser in $\mathcal D$ for disk diagrams with the combinatorial type of the standard figure-8 knot diagram violates the non-overlapping condition. This configuration of disks does not lead to a ribbon knot diagram since the associated ribbon loop violates the separation condition $||p_0-p_1||<2$, and consequently, the ribbon violates the non-overlapping condition, see Fig. \ref{fig:knotfig-8}. Observe that one segment of the ribbon overlaps another in a coloured wedge in Fig. \ref{fig:knotfig-8} (left), even though the corresponding core components do not intersect one another. A third ribbon segment lies above one of the two segments and beneath the other. The non-overlapping condition places an additional constraint on the disk diagram required for it to be ribbon. 

A conjectured minimiser that satisfies the non-overlapping condition is depicted in Fig. \ref{fig:knotfig-8} at  right, where the centres of the disks $D_0, D_1, D_2, D_3$ and $D_4$ are positioned at the points $(0, 0), (-\sqrt{2}, -\sqrt{2}), (0, -2\sqrt{2}), (\sqrt{2}, -\sqrt{2})$ and $(0, 2)$, respectively.


\begin{figure} [h!]
\centering
\includegraphics[width=\textwidth]{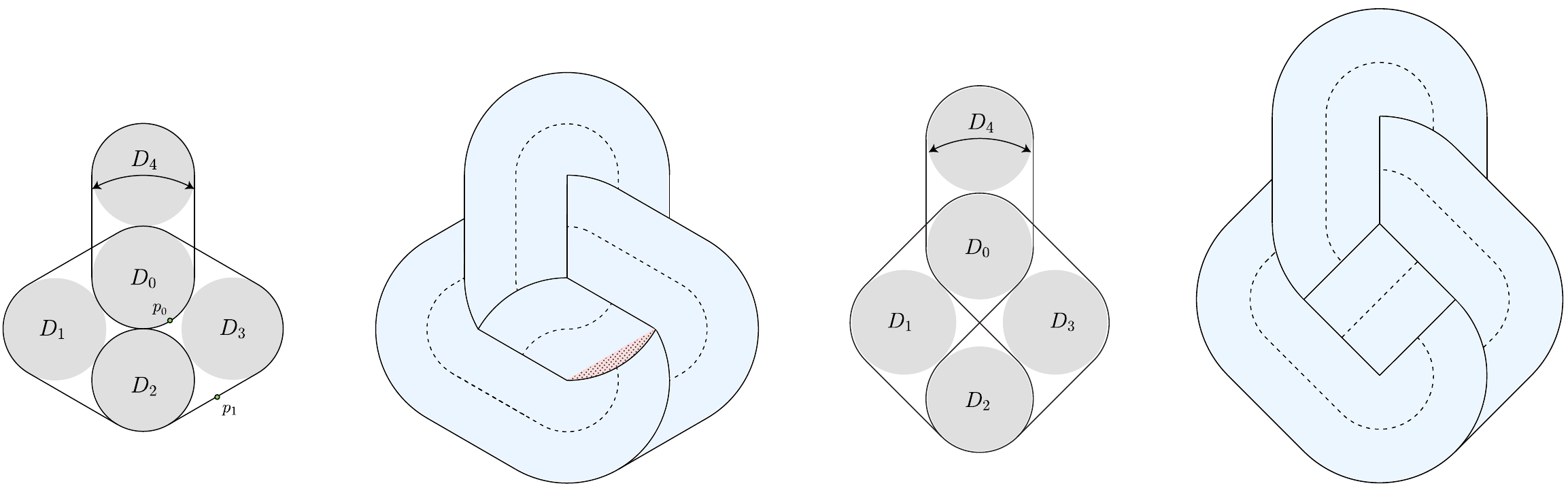}
\caption{From left to right: a length minimising disk diagram in $\mathcal D$ with the combinatorial type of the standard figure-8 knot diagram. The ribbon associated with the core is at the left. The coloured wedge illustrates an overlap between two branches of the ribbon. A conjectured length minimising ribbon loop for the standard figure-8 knot diagram on the right. The conjectured ribbon knot associated with the ribbon loop on the right.}
\label{fig:knotfig-8}
\end{figure}
\end{observation}


\subsection{Minimal length for the standard Whitehead link diagram in disk space}\label{whitehead}\hfill 


\begin{figure} [h!]
\centering
\includegraphics[width=1\textwidth]{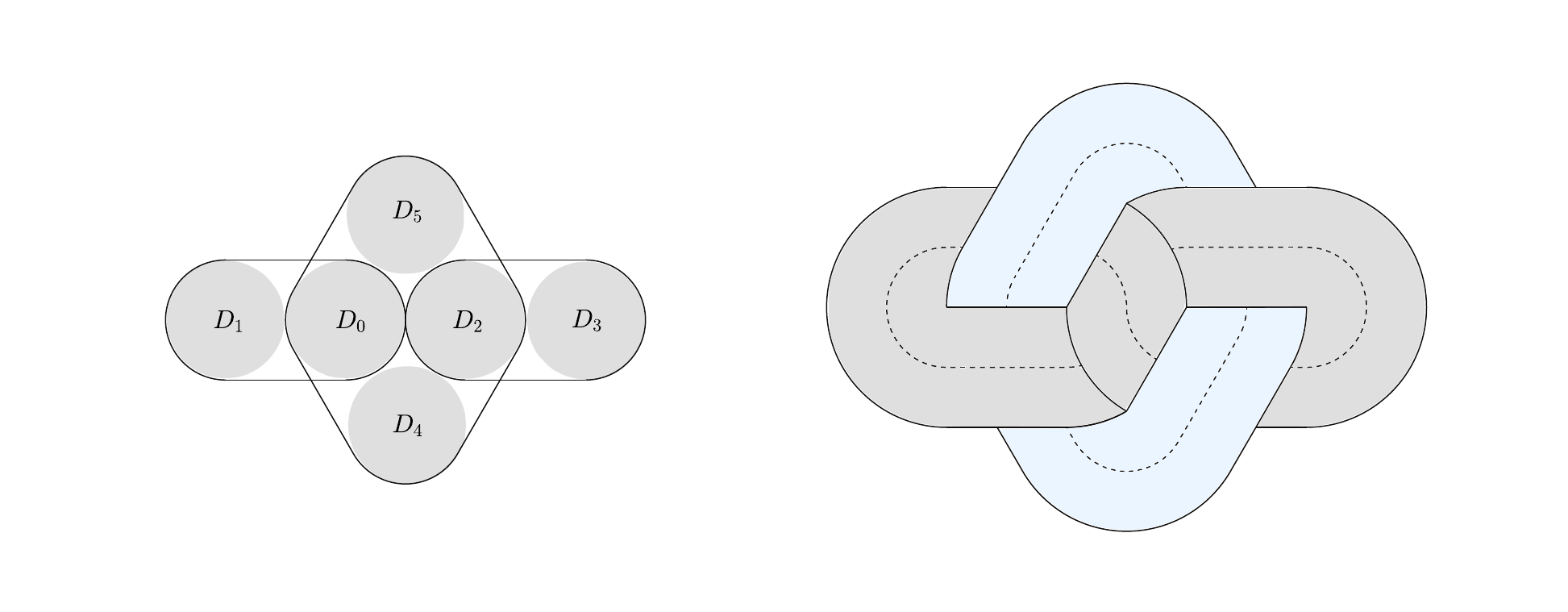}
\label{fig:Whitehead}
\caption{Left: A length minimising core for the standard Whitehead link diagram in disk space. Right: the corresponding ribbon for the standard Whitehead link diagram, violating the non-overlapping condition.}

\label{fig:knotwhitehead1}
\end{figure}
 In the case of the standard Whitehead link diagram, we may apply a similar decomposition argument but we again run into the non-overlapping condition. That is, we can decompose the core $\gamma$ into three loops, two of which contain two disks and one of which contains four disks. By a variational argument, it is easy to see that the respective lengths of the loops must be at least $4 + 2\pi, 4 + 2\pi$ and $8 + 2\pi$ due to the number of disks contained inside each of the three loops. Since the configuration depicted in Fig. 12 left attains all three bounds, it gives a minimum length element in disk space $\mathcal D$ for disk diagrams of the combinatorial type of the standard Whitehead link diagram. The corresponding ribbon is illustrated in Fig. 12. Note there is a 2-parameter family of length minimisers that is obtained by rotating the disk $D_1$ about the disk $D_0$ and the disk $D_3$ about the disk $D_2$.
Since $\mbox{Length}(\gamma)= 16 + 6\pi$, the minimium length of the standard Whitehead link diagram in disk space is $\mbox{Rib}(\mbox{Whitehead link})=  8 + 3\pi$.

\subsection{Minimal length in disk space for diagram family 1}\hfill

 Each member of the infinite family of link diagrams $F_1$, an example of which is depicted in Fig. \ref{fig:whiteheadfamily2}, comprises $n$ Whitehead links concatenated via nugatory crossings. For this choice of knot diagram as in Fig. \ref{fig:whiteheadfamily2}, we show via a decomposition argument that the minimum length of the family $F_1$ of link diagrams in disk space is
\begin{equation*}
\mbox{Rib}(F_1) =  \mbox{Rib}(\mbox{Whitehead link}) \times n = (8 + 3\pi)n.
\end{equation*}
To see this, observe that the core can be decomposed into $2n$ loops, each of which contains two disks, in addition to $n$ loops, each of which contains four disks. By a variational argument, a loop containing two disks must be of length at least $4 + 2\pi$, whereas a loop containing four disks must be of length at least $8 + 2\pi$. Since the geometric realisation of the link diagram depicted in the Fig. \ref{fig:whiteheadfamily2} attains the above bounds, it must give a minimum for the length of this particular diagram in disk space.

\begin{figure} [h!]
\centering
\includegraphics[width=1\textwidth]{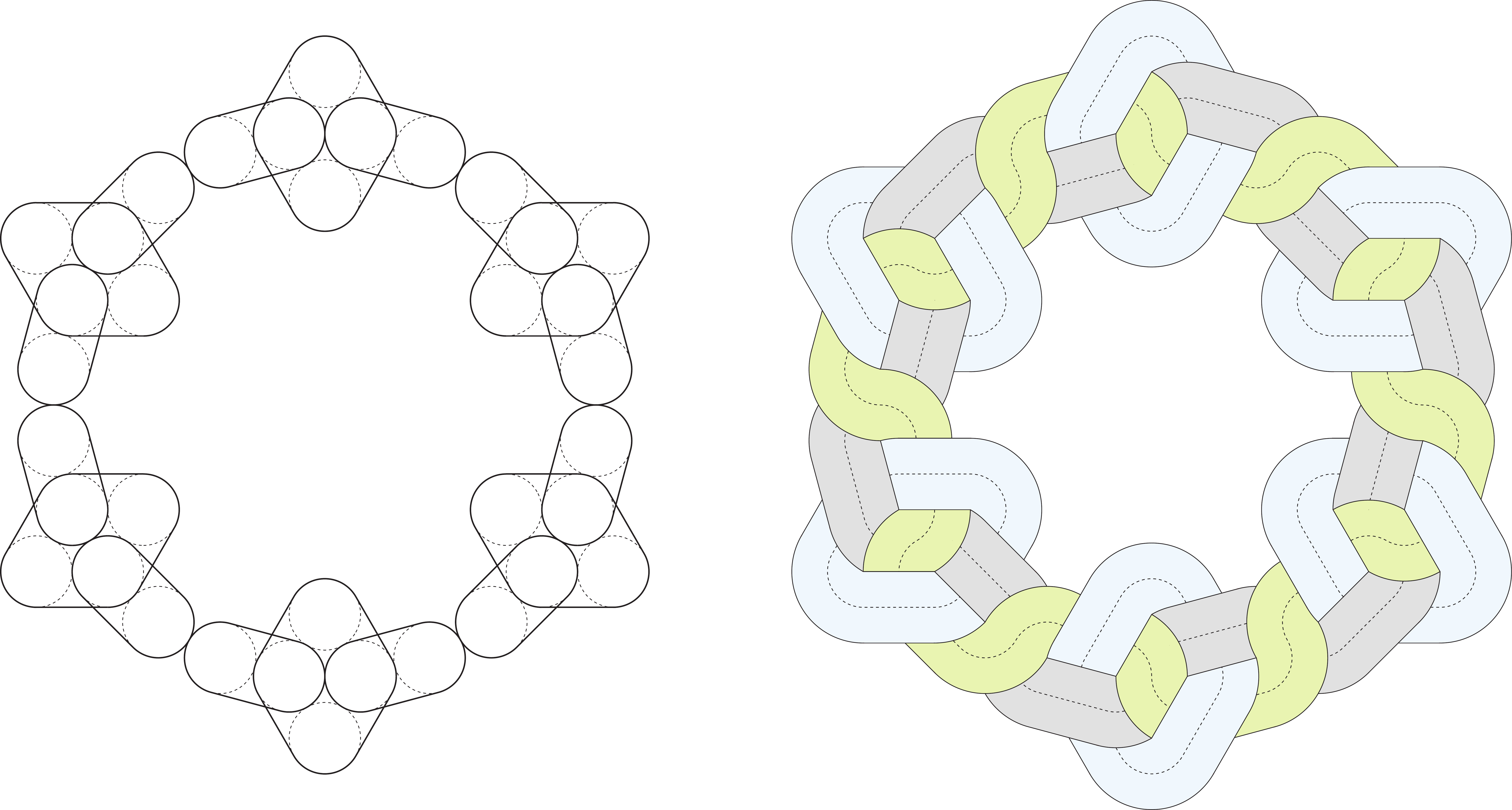}
\caption{Here the core does not satisfy the separation bound. So the ribbon is overlapping and this only gives the minimum length in disk space and is not a ribbon link diagram}
\label{fig:whiteheadfamily2}
\end{figure}

\subsection{Minimal ribbonlength for diagram family 2} \hfill

 Each member of the infinite family of link diagrams $F_2$, an example of which is depicted in Fig. \ref{fig:HopfBGG}, comprises $n$ Hopf links concatenated via nugatory crossings. Observe that the core can be decomposed into $2n$ loops, each of which contains two disks. By following a pattern of argument similar to that provided for the previous example, it is easy to see that the minimum length of the family of link diagrams $F_2$ amongst all disk diagrams is
\begin{equation*}
\mbox{Rib}(F_1) = \mbox{Rib}(\mbox{Hopf link}) \times n = (4 + 2\pi)n.
\end{equation*}


\begin{figure} [h!]
\center
\includegraphics[width=1\textwidth]{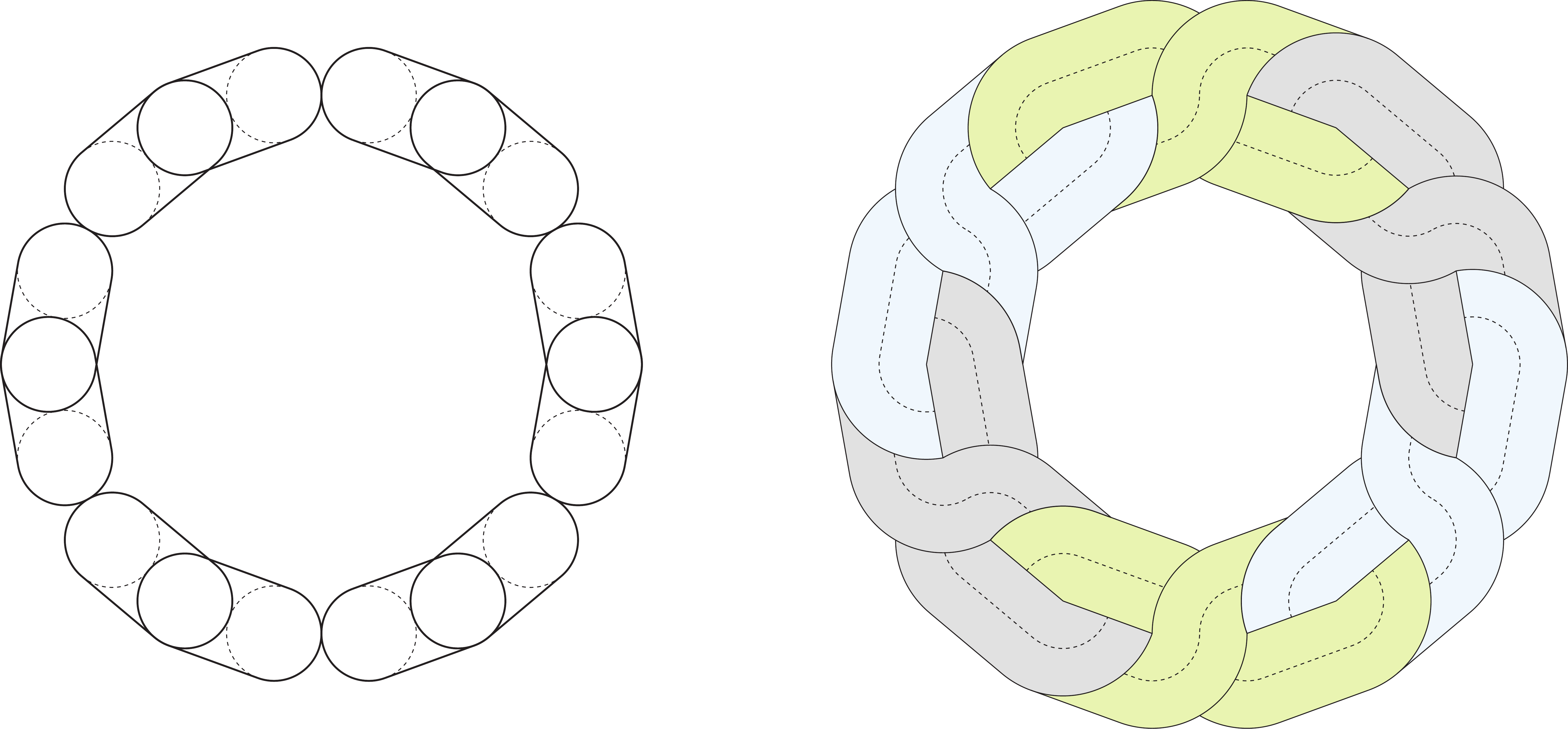}
\caption{Here the core obeys the separation bound, so the ribbon satisfies the non-overlapping condition. So this gives the minimum ribbon length diagrams}
\label{fig:HopfBGG}
\end{figure}

\section{Bounding the crossing number of minimal length knot and link ribbon and disk diagrams}

In this section, the aim is to give a useful bound on the crossing number of a projection of a knot or link which minimises ribbonlength amongst ribbon diagrams or length amongst disk diagrams. In this way, if the procedures of the previous sections could be made more efficient, then a practical method would follow for computing minimum ribbon length amongst all ribbon diagrams of a given knot or link, or the minimum length amongst all disk diagrams of a fixed knot or link. These both give (geometrical) invariants of knots and links, but the difficulty is to have a practical process for computing these minimum lengths. To illustrate the theory, we bound the crossing numbers for the trefoil knot and the Hopf link. But only for the trefoil knot, is the bound strong enough to immediately conclude that the minimum ribbon length can be computed over all diagrams. For the Hopf link, all diagrams with at most $4$ crossings need to be examined. 

\begin{theorem}\label{bound}

Suppose that $K$ is a knot or link in $\mathbb R^3$ and $\gamma$ is a ribbon (respectively disk) planar projection with length $\ell$. Then a planar projection of $K$ which minimises ribbon length (respectively length) amongst all ribbon diagrams (respectively disk diagrams) has at most $$[\frac{\ell}{\pi} -\frac{1}{2}\bigg( \sqrt{(1+\frac{4\ell}{\pi})}-1\bigg)]$$ crossings. 

\end{theorem}

\begin{proof}

Suppose  a planar projection $\gamma^*$ is ribbon (respectively disk) with $c$ transverse crossings and with minimum length amongst all ribbon (respectively disk) diagrams. Note we will slightly perturb any planar diagram to only have transverse crossings, eliminating any non-transverse crossings and replacing a crossing consisting of an interval of double points by a single transverse point or the empty set. See Figure \ref{fig:ribcross}, where the crossings in the right figures three and four get eliminated. Then $\gamma^*$ forms a graph $G$ with $c$ vertices of degree $4$ and hence $2c$ edges joining the vertices. 

Since the Euler characteristic of the plane is one, we see that $\gamma^*$ must have $c+1$ bounded complementary regions. Now by Proposition \ref{disk} if $\gamma^*$ is a ribbon projection, then each bounded  complementary region contains a unit disk, i.e the projection is a disk diagram.  Hence each of the bounded complementary regions has area at least at least $\pi$. By the isoperimetric inequality, each bounded complementary region has boundary length at least $2\pi$. Adding up the lengths of the boundaries of the bounded complementary regions gives at least $2(c+1)\pi$.

Let us denote by $\partial G$ the subgraph of $G$ consisting of edges which are in the boundary of a single bounded complementary region of $\gamma$. The edges of $\hat G=G \setminus \partial G$ will be referred to as {\it internal} edges of $G$ and clearly are in the boundary of two bounded complementary regions of $\gamma$.  Hence the sum of the lengths of the boundaries of the bounded complementary regions is $\| \partial G\|+2\|\hat G\|$. Combining, we get $2\ell=2\|\gamma\|\ge 2\|\gamma^*\|=2\|G\| =2 \| \partial G\|+2\|\hat G\| \ge 2(c+1)\pi + \|\partial G\|$. 

Next,  we know that the area enclosed by $\partial G$ is at least $(c+1)\pi$. By another application of the isoperimetric inequality, this implies $\|\partial G\| \ge 2\sqrt{(c+1)\pi}$. Putting this together with the inequality above, we get $2\ell \ge 2(c+1)\pi +2\sqrt{(c+1)\pi}$. This can be rewritten as $\frac{\ell}{\pi} \ge c+1 +\sqrt{(c+1)}$. Now this quadratic inequality is easily solved giving the final result $c \le \frac{\ell}{\pi} -\frac{1}{2}\bigg( \sqrt{(1+\frac{4\ell}{\pi})}-1\bigg)$.

\end{proof}

\begin{example}

For the standard projection of the trefoil knot, we found in \ref{trefoil} that the minimum ribbon length is $6+2\pi$. So putting twice this length into Theorem \ref{bound} 

$$c \le \frac{12+4\pi}{\pi}-\frac{1}{2}\bigg( \sqrt{1+\frac{4(12+4\pi)}{\pi}}-1\bigg)\approx 5.47$$

So to verify that the minimum ribbon length is that of the standard projection, it suffices to check diagrams of the trefoil with at most $5$ crossings. 


\end{example}

\begin{example}

For the standard projection of the Hopf link, we found in \ref{hopf} that the minimum ribbon length is $4+2\pi$. So putting twice this length into Theorem \ref{bound} gives 

$$c \le \frac{8+4\pi}{\pi}-\frac{1}{2}\bigg( \sqrt{1+\frac{4(8+4\pi)}{\pi}}-1\bigg)\approx 4.43$$

So to verify that the minimum ribbon length is that of the standard projection, it suffices to check diagrams of the Hopf links with at most $4$ crossings. The standard diagram has $2$ crossings. 

\end{example}

\section{Discussion}

There are three different classes of planar knot and link diagrams considered in this paper. The first class are ribbon knots and links which satisfy the separation bound and crossing condition. The second are diagrams which lie in disk space, i.e there is a unit disk in each complementary bounded region. Finally, a link which is a finite collection of ribbon knot diagrams gives a weaker condition than for ribbon links. So this gives a third class of diagrams. 

However, to obtain the crucial Proposition \ref{disk} it is clear that the weaker property of a link being a finite union of ribbon knots is not sufficient. Hence we cannot embed our problem into disk space using this weaker property. The advantage of using disk space is Theorem \ref{ribcs}. This result shows that if a disk space length minimiser is ribbon, then this diagram is a concatenation of arcs of circles of radius one and straight line segments. 

Note also that the proof of Proposition \ref{disk} uses an important property of ribbon knots and links, namely that the ribbon deformation retracts onto the core of the knot or link. This latter property is implied by the crossing condition. In fact, we chose the crossing condition with this property in mind.

  In work in progress, the first and third authors will introduce a new mathematical model of physical knots and links. The current paper is related to this project in that we have tried to produce a model of ribbon knots and links in the plane with similar properties to the $3$-dimensional model, which is harder to solve. However, two nice features of the $3$-dimensional model are that `ropelength' minimisers always exist and have good regularity properties for any knot or link type. Moreover for the Sussmann problem, see \cite{sus}, namely finding a shortest length  path between two nearby points with initial and final directions not differing too much, in our $3$-dimensional model the solution is a concatenation of arcs of circles of radius one and straight line segments. However we are not able to compute length minimisers for any knots and links, except for the unknot. 

Next, the bound obtained in Theorem \ref{bound} is most likely not optimal. It would be interesting to refine the bound, using more properties of ribbon and disk space length minimisers. The proof of Theorem \ref{bound} only uses the isoperimetric inequality. By bringing the bound closer to the minimum crossing number, e.g for alternating knots and links, this would give a much more efficient approach to computing the minimum ribbon and disk length over all planar diagrams.


{}

\begin{thebibliography}{}


%
%

\bibitem {ayala} J. Ayala, D. Kirszenblat, J.H. Rubinstein,  A geometric approach
to shortest bounded curvature paths, Communications in Analysis
and Geometry, 26, (4), (2018), 679--697.


\bibitem {sullivan1} J. Cantarella, R.B. Kusner, J.M. Sullivan, On the minimum ropelength of knots and links, Invent. Math. 150 (2002) 257--286.

\bibitem {denne1} E. Denne, Folded ribbon knots in the plane, arXiv:1807.00691v1.


\bibitem{denne3} E. Denne, J. M. Sullivan and N. C. Wrinkle, The ribbonlength of knot diagrams, Preliminary report 1132-54-316, \url{https://www.ams.org/amsmtgs/2240_abstracts/1132-54-316.pdf}

\bibitem {diao}Y. Diao, C. Ernst, and E. J. Janse van Rensburg,  Thickness of knots, Math. Proc. Camb. Phil. Soc. 126, (1999), 293--310.

\bibitem {gonzalez} O. Gonzalez, J.H. Maddocks,  Global curvature, thickness, and the ideal shapes of knots. Proc. Nat. Acad. Sci. (USA) 96, (1999), 4769--4773.

\bibitem {kauffman} L.H. Kauffman, Minimal flat knotted ribbons, Physical and numerical models in knot theory, Singapore, World Sci. Publ. (2005), Vol 36 Ser. Knots Everything, 495--506.

\bibitem{kirszenblat} D. Kirszenblat and K. G. Sirinanda, M. Brazil, P. A. Grossman, J. H. Rubinstein and D. A. Thomas, Minimal curvature-constrained networks, J. Global Optimization 72 (2018), 71--87.

\bibitem{litherland} R.A. Litherland, J. Simon, O. Durumeric, E. Rawdon, Thickness of Knots, Topology and its Applications 91, (1999), 233--244.

\bibitem{meeks} W. Meeks and S.T. Yau, Topology of three dimensional manifolds and the embedding problems in minimal surface theory. Annals of Mathematics. 112 (3): (1982), 441--484.

 \bibitem{munkres}  J. Munkres, Topology (2nd ed.) Prentice Hall.,(1999) pp. 280--281.

\bibitem{rubthom} J. H. Rubinstein, D. A. Thomas, A variational approach to the Steiner ratio conjecture, 33, Ann Oper Res (1991)  481--499. https://doi.org/10.1007/BF02071984

\bibitem{sus} H. Sussmann, Shortest 3-dimensional paths with a prescribed curvature bound, Proceedings of the 34th IEEE conference on Decision and Control, 1995, 


\bibitem{younes} L. Younes. Shapes and Diffeomorphisms. Springer, 2010.

\end{thebibliography}
\end{document}